\documentclass[11pt]{amsart}

\usepackage{amssymb,amsfonts}
\usepackage[all,arc]{xy}
\usepackage{enumerate}
\usepackage{mathrsfs}
\usepackage{eufrak}
\usepackage{setspace}
\usepackage{color}
\usepackage{graphicx}
\usepackage{pdfpages}
\usepackage{enumitem}

\newtheorem{theorem}{Theorem}[section]
\newtheorem{proposition}[theorem]{Proposition}

\newtheorem{corollary}[theorem]{Corollary}
\newtheorem{definition}[theorem]{Definition}

\newtheorem{example}[theorem]{Example}

\theoremstyle{remark}
\newtheorem{remark}[theorem]{Remark}

\renewenvironment{proof}{{\noindent\bf Proof.}}{\hfill $\Box$\par\vskip3mm}

\newcommand{\Aa}{\mathcal{A}}
\newcommand{\Bb}{\mathcal{B}}
\newcommand{\Cc}{\mathcal{C}}

\newcommand{\Ff}{\mathcal{F}}

\def\AA{{\mathbb A}}

\theoremstyle{definition}

\theoremstyle{remark}

\makeatletter
\let\c@equation\c@thm
\makeatother
\numberwithin{equation}{section}

\title{$n$-quivers and A Universal Investigation of $n$-representations of Quivers }
\author{ADNAN H. ABDULWAHID}
\date{}

\begin{document}

\begin{abstract}
\noindent  We have two parallel goals of this paper. First, we investigate and construct cofree coalgebras over  $n$-representations of quivers,  limits and colimits of $n$-representations of quivers, and limits and colimits of coalgebras in the monoidal categories of $n$-representations of quivers.  \\
Second, we introduce a generalization for quivers and show that this generalization can be seen as essentially the same as $n$-representations of quivers. This is significantly important because it allows us to build a new quiver $\mathscr{Q}_{\!_{(\mathit{Q}_1, \mathit{Q}_2,...,  \mathit{Q}_n)}}$, called $n$-quiver, from any given quivers $\mathit{Q}_1,\mathit{Q}_2,...,\mathit{Q}_n$, and enable us to identify each category $Rep_k(\mathit{Q}_j)$ of representations of a quiver $\mathit{Q}_j$ as a full subcategory of the category $Rep_k(\mathscr{Q}_{\!_{(\mathit{Q}_1, \mathit{Q}_2,...,  \mathit{Q}_n)}})$ of representations of $\mathscr{Q}_{\!_{(\mathit{Q}_1, \mathit{Q}_2,...,  \mathit{Q}_n)}}$ for every $j \in \{1,\,\,2,\,\,..., \,\,n\}$. \\  
\end{abstract}

\date{}

\maketitle


\noindent

\section{Introduction}
\label{intro}

The quiver representations theory has been developed exponentially since it was introduced in \cite{Gabriel} due to Gabriel in 1972.
Very recently, it has significantly invaded a large area of substantial studies, and it has been seen as a pivot source of inspiration for many theories. Due to its combinatorial flavor,
it has become as a vital subject with germane connections to  associative algebra, combinatorics, algebraic topology, algebraic geometry, quantum groups, Hopf algebras, tensor categories. As it is well-known, there is a forgetful functor, which has a left adjoint,  from the category of small categories to the category of quivers. I turns out that the quiver representations theory supports a very positive connection between combinatorics and category theory, and this allows to substantially distinctive theories to collaboratively work together as a team to reshape many theories and build them in terms of modern techniques. \\

As an important strategy for developing quiver representations theory, we need to carefully and systematically generalize the notions of quivers and their representations. Our main aim is to introduce  generalizations for quivers and their representations. Unlike one might expect, our generalization approach is to start with introducing a generalization for representations of quivers. Then we investigate universal properties and constructions for $n$-representations of quivers. Finally, we  introduce a generalization for quivers and prove that this generalization and our previous generalization approach are essentially the same in the sense of a categorical point of view. \\

The first part of our aim has already been done in \cite{Abdulwahid1} by introducing the concept of $n$-representations of quivers, providing concrete examples of such notions, establishing categories of $n$-representations of quivers, and showing that these categories are abelian. For any $n \geq 2$,  an $n$-representation of an $n$-tuple $(\mathit{Q}_1,\mathit{Q}_2,...,\mathit{Q}_n)$ of quivers is a generalization of a representation of $\mathit{Q}_j$, for any $j \in \{1,\,\,2,\,\,..., \,\,n\}$. Further, the definition of $n$-representations of quivers allows us to identify each category $Rep_k(\mathit{Q}_j)$ of representations of a quiver $\mathit{Q}_j$ as a full subcategory of the category $Rep_{\!_{(\mathit{Q}_1,\mathit{Q}_2,...,\mathit{Q}_n)}}$ of $n$-representations of $(\mathit{Q}_1,\mathit{Q}_2,...,\mathit{Q}_n)$ for every $j \in \{1,\,\,2,\,\,..., \,\,n\}$. \\

This paper is mainly devoted to two parallel goals. The fist one is to
investigate cofree coalgebras over  $n$-representations of quivers,  limits and colimits of $n$-representations of quivers, and limits and colimits of coalgebras in the categories of $n$-representations of quivers. Moreover, we explicitly construct cofree coalgebras, limits, colimits in these (monoidal) categories. \\
The other goal of this paper is to introduce a generalization for quivers and prove that this generalization can be recast into $n$-representations of quivers. This is crucially important since for any given quivers $\mathit{Q}_1,\mathit{Q}_2,...,\mathit{Q}_n$, it allows us to build a new quiver $\mathscr{Q}_{\!_{(\mathit{Q}_1, \mathit{Q}_2,...,  \mathit{Q}_n)}}$, called $n$-quiver, by which we are able to view each category $Rep_k(\mathit{Q}_j)$ of representations of a quiver $\mathit{Q}_j$ as a full subcategory of the category $Rep_k(\mathscr{Q}_{\!_{(\mathit{Q}_1, \mathit{Q}_2,...,  \mathit{Q}_n)}})$ of representations of $\mathscr{Q}_{\!_{(\mathit{Q}_1, \mathit{Q}_2,...,  \mathit{Q}_n)}}$ for every $j \in \{1,\,\,2,\,\,..., \,\,n\}$.\\

To formulate our goal concerned with a universal investigation, we need to recall some categorical definitions.\\   

Let $\mathfrak{X}$ be a category. A \textit{concrete category} over $\mathfrak{X}$ is a pair $(\mathfrak{A},\mathfrak{U})$, where $\mathfrak{A}$ is a category and $\mathfrak{U}: \mathfrak{A} \rightarrow \mathfrak{X}$ is a faithful functor \cite[p. 61]{Adamek}. Let $(\mathfrak{A},\mathfrak{U})$ be a concrete category over $\mathfrak{X}$. Following \cite[p. 140-143]{Adamek}, a \textit{free object} over $\mathfrak{X}$-object $X$ is an $\mathfrak{A}$-object $A$ such that there exists a \textit{universal arrow}  $(A,u)$ over $X$; that is, $u:X \rightarrow \mathfrak{U}A$ such that for every arrow $f:X \rightarrow \mathfrak{U}B$, there exists a unique morphism $f':A \rightarrow B$ in $\mathfrak{A}$ such that $\mathfrak{U}f' u=f$. We also say that $(A,u)$ is the free object over $X$.  A concrete category  $(\mathfrak{A},\mathfrak{U})$ over $\mathfrak{X}$ is said to \textit{have free objects} provided that for each $\mathfrak{X}$-object $X$, there exists a universal arrow over $X$. For example, the category $Vect_{\mathbb{K}}$ of vector spaces over a field $\mathbb{K}$ has free objects. So do the category \textbf{Top} of topological spaces and the category of \textbf{Grp} of groups. However, some interesting categories do not have free objects \cite[p. 142]{Adamek}). \\

Dually, \textit{co-universal arrows}, \textit{cofree objects} and categories that \textit{have cofree objects} can be defined. For the basic concepts of concrete categories, free objects, and cofree objects, we refer the reader to \cite[p. 138-155]{Kilp}. \\

It turns out that a  concrete $(\mathfrak{A},\mathfrak{U})$ over $\mathfrak{X}$ has (co)free objects if and only if the functor that builds up (co)free object is a (right) left adjoint to the faithful functor $\mathfrak{U}: \mathfrak{A} \rightarrow \mathfrak{X}$.\\

Therefore, our goal involved with a universal investigation can explicitly be formulated as follows. Let    $\mathscr{U}_{n}: CoAlg(Rep_{\!_{(\mathit{Q}_1,\mathit{Q}_2,...,\mathit{Q}_n)}}) \rightarrow Rep_{\!_{(\mathit{Q}_1,\mathit{Q}_2,...,\mathit{Q}_n)}}$ be the forgetful functor from the category of coalgebras in the category $Rep_{\!_{(\mathit{Q}_1,\mathit{Q}_2,...,\mathit{Q}_n)}}$ to the category $Rep_{\!_{(\mathit{Q}_1,\mathit{Q}_2,...,\mathit{Q}_n)}}$. Does  $\mathscr{U}_{n}$ have a right adjoint?   \\

An expected strategy for the answer of this question is to use  
the dual of Special Adjoint Functor Theorem (D-SAFT). \\

We start our inspection by showing that the limits and colimits in the categories of $n$-representations of quivers can be inherited from their respective limits and colimits of representations of quivers. Then we introduce the notion of $n$-quivers which suggests that there is an equivalence of categories between the categories of $n$-representations of quivers and the categories of representations of $n$-quivers. Eventually, we show that cofree coalgebras exist in the monoidal categories of $n$-representations of quivers and describe them in terms of cofree coalgebras in the monoidal categories of  representations of quivers.\\
 
The sections of this paper can be summarized as follows.\\

In Section 2, we give some detailed background on quiver representations and few categorical notions that we need for the next sections. As a preliminary step, we briefly discuss the basics about $n$-representations of quivers. For more details and concrete examples of $n$-representations of quivers, we refer the reader to \cite{Abdulwahid1}.\\ 

In Section 3,  we briefly discuss the basics about $n$-representations of quivers. \\ 

In Section 4, we start with an investigation of limits of birepresentations ($2$-representations of quivers) and inductively extend our results to limits of $n$-representations of quivers. Similarly to the case of limits, we investigate and construct colimits of $n$-representations of quivers. \\

In Section 5, we introduce the notion of $2$-quivers and inductively generalize the notion for of $n$-quivers. We explicitly give concrete examples of $n$-quivers and representations of $n$-quivers. Further, we show that the concept of $n$-quivers gives rise to identify 
the categories of $n$-representations of quivers and the categories of representations of $n$-quivers as essentially the same. This  identification allows one to have an explicit description for the generators of the categories of $n$-representations of quivers and characterize many properties of $n$-representations of quivers. \\
Finally, we investigate cofree coalgebras in the monoidal categories of $n$-representations of quivers. We also construct them in terms of colimits and generators and show that cofree coalgebras in these monoidal categories can explicitly be obtained from cofree coalgebras in the monoidal categories of quiver representations. \\

\section{\textbf{Preliminaries}}\label{s.p}
Throughout this paper $k$ is an algebraically closed field, $n \geq 2$,  and $\mathit{Q}, \,\, \mathit{Q}', \,\, \mathit{Q}_1, \,\, \mathit{Q}_2, ..., \,\, \mathit{Q}_n$ are  quivers. We also denote $k\mathit{Q}$ the path algebra of $\mathit{Q}$. Unless otherwise specified, we will consider only finite, connected, and acyclic quivers. \\
Let $\mathcal{A}$ be a (locally small) category and $A$, $B$  objects in $\mathcal{A}$. We denote by $\mathcal{A} (A,B)$ the set of all morphisms from $A$ to $B$. \\
Let $\Aa$, $\Bb$ be categories. Following \cite[p. 74]{McLarty},  the  \textbf{product category }$\Aa \times \Bb$ is the category whose objects are all pairs of the form $(A,B)$, where $A$ is an object of $\Aa$ and $B$ an object of $\Bb$. An arrow is a pair $(f,g): (A,B) \rightarrow (A',B')$, where $f: A \rightarrow A'$ is an arrow of $\Aa$ and $g: B \rightarrow B'$ is an arrow of $\Bb$. The identity arrow for $\Aa \times \Bb$ is $(id_{\!_{A}},id_{\!_{B}})$ and composition is defined component-wise, so $(f,g)(f',g') = (ff',gg')$.There is a functor $P_1:\Aa \times \Bb \rightarrow \AA$ defined by $P_1(A,B)= A $
and $P_1 (f,g) =f$, and an obvious $P_2$.  
We will need the following.

\begin{theorem} \cite[p. 148]{Freyd} \label{p.SAFT} 
If $\mathfrak{A}$ is cocomplete, co-wellpowered and with a generating set, then every cocontinuous functor from $\mathfrak{A}$ to a locally small category has a right adjoint.
\end{theorem}

\begin{proposition}\label{p.cowp}
Let $(\Cc,\otimes,I)$ be a monoidal category, $CoMon(\Cc)$ be the category of comonoids of $\Cc$ and $U:CoMon(\Cc)\rightarrow \Cc$ be the forgetful functor. \\
(i) If $\Cc$ is cocomplete, then $CoMon(\Cc)$ is cocomplete and $U$ preserves colimits. \\
(ii) If furthermore $\Cc$ is co-wellpowered, then so is $CoMon(\Cc)$. \\
\end{proposition}

For the fundamental concepts of category theory, we refer to \cite{Leinster1}, \cite{Mac Lane1}, \cite{Awodey}, \cite{Rotman}, \cite{Pareigis}, or  \cite{Mitchell}.\\

Following \cite{Schiffler}, a \textbf{quiver} $\mathit{Q} = (\mathrm{Q}_{\!_{0}}, \mathrm{Q}_{\!_{1}}, s,t)$  consists of
\begin{itemize}
\item $\mathrm{Q}_{\!_{0}}$ a set of vertices,
\item $\mathrm{Q}_{\!_{1}}$ a set of arrows,
\item $s:\mathrm{Q}_{\!_{1}} \rightarrow \mathrm{Q}_{\!_{0}}$ a map from arrows to vertices, mapping an arrow to its starting
point,
\item $t:\mathrm{Q}_{\!_{1}} \rightarrow \mathrm{Q}_{\!_{0}}$ a map from arrows to vertices, mapping an arrow to its terminal point.
\end{itemize}
We will represent an element $\alpha \in \mathrm{Q}_{\!_{1}}$ by drawing an arrow from its starting point $s(\alpha)$ to its endpoint $t(\alpha)$ as follows:
$s(\alpha) \xrightarrow{\alpha} t(\alpha)$.\\
A \textbf{representation} $M = (M_i,\varphi_{\alpha})_{i \in \mathrm{Q}_{\!_{0}}, \alpha \in Q{\!_{1}}}$ of a quiver $Q$ is a collection of $k$-vector spaces $M_i$ one for each vertex $i \in \mathrm{Q}_{\!_{0}}$, and a collection of $k$-linear maps $\varphi_{\alpha}: M_{s(\alpha)} \rightarrow M_{t(\alpha)}$ one for each arrow $\alpha \in \mathrm{Q}_{\!_{1}}$. \\
A representation $M$ is called \textbf{finite-dimensional} if each vector space Mi is finite-dimensional.\\
Let $Q$ be a quiver and let $M = (M_i,\varphi_{\alpha})$, $M' = (M'_i,\varphi'_{\alpha})$  be two representations of $Q$. A \textbf{morphism} of representations $f : M \rightarrow M'$ is a collection $(f_i)_{i \in Q{\!_{0}}}$ of $k$-linear maps $f_i : M_i \rightarrow M'_i$
such that for each arrow $s(\alpha) \xrightarrow{\alpha} t(\alpha)$ in 
 $\mathrm{Q}_{\!_{1}}$ the diagram 
 \begin{equation} \label{diag.eq01}
\xymatrix{
M_{s(\alpha)} \ar[rr]^{\phi_{\alpha}} \ar[d]_{f_{s(\alpha)}} && M_{t(\alpha)} \ar[d]^{f_{t(\alpha)}}\\
M'_{s(\alpha)} \ar[rr]_(.5){\phi'_{\alpha}} && M'_{t(\alpha)}
} 
\end{equation}
commutes. \\
A morphism of representations $f = (f_i): M \rightarrow M'$  is an equivalence if each $f_i$ is bijective. The class of all representations that are isomorphic to a given representation $M$ is called the \textbf{isoclass} of $M$.\\
This gives rise to define a category $Rep_k(\mathit{Q})$ of $k$-linear representations of  $\mathit{Q}$. We denote by $rep_k(\mathit{Q})$ the full  subcategory of $Rep_k(\mathit{Q})$ consisting of the finite dimensional representations.

Given two representations $M = (M_i,\phi_{\alpha})$ and  $M' = (M'_i,\phi'_{\alpha'})$ of $\mathit{Q}$, the representation

\begin{equation} \label{diag.eq001}
M \oplus M' \,\,\ = \,\,\ (M_i \oplus M'_i,\begin{bmatrix}
       \phi_{\alpha} & 0            \\[0.3em]
       0 & \phi'_{\alpha}         \\[0.3em] \end{bmatrix})
\end{equation}
is the \textbf{direct sum} of $M$ and $M'$ in $Rep_k(\mathit{Q})$ \cite[p. 71]{Assem}. \\
A nonzero representation of a quiver $\mathit{Q}$ is said to be \textbf{indecomposable} if it is not isomorphic to a direct sum of two nonzero representations \cite[p. 21]{Etingof1}. \\

Following \cite[p. 114]{Schiffler}, the \textbf{path algebra} $k\mathit{Q}$ of a quiver $\mathit{Q}$ is the algebra with basis the set of all paths in the quiver $\mathit{Q}$ and with multiplication defined on two basis elements $c, c'$  by

\[
    c.c' = 
\begin{cases}
    cc',& \text{if } s(c') =  t(c)\\
    0,              & \text{otherwise.}
\end{cases}
\]

We will need the following propositions.

\begin{proposition} \label{p.1} \cite[p. 70]{Assem}
Let $\mathit{Q}$ be a finite quiver. Then $Rep_k(\mathit{Q})$ and $rep_k(\mathit{Q})$
are $k$-linear abelian categories.
\end{proposition}

\begin{proposition} \label{p.Equiv.Rep.Mod} \cite[p. 74]{Assem}
Let $\mathit{Q}$ be a finite, connected, and acyclic quiver. There
exists an equivalence of categories $Mod \, k\mathit{Q} \simeq Rep_k(\mathit{Q})$ that restricts to an equivalence $mod \, k\mathit{Q} \simeq rep_k(\mathit{Q})$, where $k\mathit{Q}$ is the path algebra of $\mathit{Q}$, $Mod \, k\mathit{Q}$ denotes the category of right $k\mathit{Q}$-modules, and  $mod \, k\mathit{Q}$ denotes the full subcategory of $Mod \, k\mathit{Q}$ consisting of the finitely generated right $k\mathit{Q}$-modules.
\end{proposition} 

We will also need the following theorem.

\begin{theorem} \cite[p. 148]{Freyd} \label{p.SAFT} 
If $\mathfrak{A}$ is cocomplete, co-wellpowered and with a generating set, then every cocontinuous functor from $\mathfrak{A}$ to a locally small category has a right adjoint.
\end{theorem}

This is a very brief review of the basic concepts involved with our work. For the basic notions of quiver representations theory, we refer the reader to \cite{Assem}, \cite{Schiffler}, \cite{Auslander1}, \cite{Barot}, \cite{Etingof1}, \cite{Benson}, \cite{Zimmermann}.\\

\vspace{.2cm}

\section{\textbf{$n$-representations of Quivers: Basic Concepts}}\label{s.n.rep.quivers}
Let $\mathit{Q} = (\mathrm{Q}_{\!_{0}}, \mathrm{Q}_{\!_{1}}, s,t)$, $\mathit{Q}' = (\mathrm{Q}'_{\!_{0}}, \mathrm{Q}'_{\!_{1}}, s',t')$ be quivers. 

\begin{definition} \cite{Abdulwahid1} \label{def.1} 
A \textbf{$2$-representation} of  $(\mathit{Q},\mathit{Q}')$  (or a \textbf{birepresentation} of $(\mathit{Q},\mathit{Q}')$) is a triple $\bar{M} = ((M_i,\phi_{\alpha}),(M'_{i'},\phi'_{\beta}), (\psi^{\alpha}_{\beta}))_{i \in \mathrm{Q}_{\!_{0}}, i' \in \mathrm{Q}'_{\!_{0}}, \alpha \in \mathrm{Q}_{\!_{1}},\beta \in \mathrm{Q}'_{\!_{1}}}$, where $(M_i,\phi_{\alpha}),(M'_{i'},\phi'_{\beta})$ are representations of $\mathit{Q},\mathit{Q}'$  respectively, and $(\psi^{\alpha}_{\beta})$ is a collection of $k$-linear maps $\psi^{\alpha}_{\beta}: M_{t(\alpha)} \rightarrow M'_{s(\beta)} $,  one for each pair of arrows $(\alpha,\beta) \in  \mathrm{Q}_{\!_{1}} \times  \mathrm{Q}'_{\!_{1}}$. \\ 

Unless confusion is possible, we denote a birepresentation simply by $\bar{M} = (M,M',\psi)$. Next, we inductively define $n$-representations for any integer $n \geq 2$. \\
For any  $m \in \{2,\,\,..., \,\,n\}$, let $\mathit{Q}_m = (\mathrm{Q}^{(m)}_{\!_{0}}, \mathrm{Q}^{(m)}_{\!_{1}}, s^{(m)},t^{(m)})$ be a quiver.  A \textbf{$n$-representation} of $(\mathit{Q}_1, \,\, \mathit{Q}_2, \,\,... \,\, ,  \mathit{Q}_n)$ is $(2n-1)$-tuple  $\bar{V} = (V^{(1)},V^{(2)},...,V^{(n)},\psi_{\!_{1}}, \psi_{\!_{2}}, ..., \psi_{\!_{n-1}})$, where for every $m \in \{1, \,\,2,\,\,..., \,\,n\}$, $V^{(m)}$ is a representation of  $\mathit{Q}_m$, and $(\psi_{\!_{m}\gamma^{(m-1)}}^{\gamma^{(m)}}) $ is a collection of $k$-linear maps 
\begin{center}
•$\psi_{\!_{m}\gamma^{(m-1)}}^{\gamma^{(m)}}: V^{(m-1)}_{t^{(m-1)}(\gamma^{(m-1)})} \rightarrow V^{(m-1)}_{s^{(m)}(\gamma^{(m)})} $,
\end{center}
  one for each pair of arrows $(\gamma^{(m-1)},\gamma^{(m)}) \in \mathrm{Q}^{(m-1)}_{\!_{1}} \times  \mathrm{Q}^{(m)}_{\!_{1}}$  and $m \in \{2,\,\,..., \,\,n\}$.   \\ 

\end{definition}

\begin{remark} \cite{Abdulwahid1} \label{r.inductive} \textbf{•}
\begin{enumerate}[label=(\roman*)]
\item When no confusion is possible, we simply write $s,t$ instead of $s',t'$ respectively, and for every  $m \in \{1, \,\,2,\,\,..., \,\,n\}$, we write $s,t$ instead of $s^{(m)},\,t^{(m)}$ respectively. 
\item It is clear that if $(V^{(1)},V^{(2)},...,V^{(n)},\psi_{\!_{1}}, \psi_{\!_{2}}, ..., \psi_{\!_{n-1}})$ is an $n$-representation of $(\mathit{Q}_1, \,\, \mathit{Q}_2, \,\,... \,\, ,  \mathit{Q}_n)$, then  $(V^{(1)},V^{(2)},...,V^{(n-1)},\psi_{\!_{1}}, \psi_{\!_{2}}, ..., \psi_{\!_{n-2}})$ is an $(n-1)$-representation of $(\mathit{Q}_1, \,\, \mathit{Q}_2, \,\,... \,\, ,  \mathit{Q}_{n-1})$ for every integer $n \geq 2$.
\item Part $(ii)$ implies that for any integer $n > 2$, $n$-representations roughly inherit all the properties and the universal constructions that  $(n-1)$-representations have. Thus, we mostly focus on studying birepresentations since they can be regarded as a mirror in which one can see a clear decription of $n$-representations for any integer $n > 2$.\\
\end{enumerate} 

\end{remark}

\begin{example} \cite{Abdulwahid1} \label{ex.1} 
Let $\mathit{Q},\mathit{Q}'$ be the following quivers 
\begin{equation} \label{diag.eq06}
\xymatrix{
&&\\
\mathit{Q}: & 1 \ar[r]^{} & 2 
}
\hspace{80pt}
\xymatrix{
&1 \ar[dr]^{} & &\\
\mathit{Q}':&&3  &4 \ar[l]^{}\\
&2  \ar[ur]^{} &&
}
\end{equation}
and consider the following:
\begin{equation} \label{diag.eq07}
\xymatrix{
&k \ar[dr]^{\begin{bmatrix}
       1            \\[0.3em]
       0         \\[0.3em]
       \end{bmatrix}} & &\\
M &&k^2  &k \ar[l]_{\begin{bmatrix}
       1            \\[0.3em]
       1         \\[0.3em]
       \end{bmatrix}}\\
&k  \ar[ur]_{\begin{bmatrix}
       0            \\[0.3em]
       1         \\[0.3em]
       \end{bmatrix}} &&
}
\hspace{70pt}
\xymatrix{
&&&\\
M' &k  &k \ar[l]_{1}
}
\end{equation}
Then $M$ (respectively $M'$) is a representation of $\mathit{Q}$ ((respectively $\mathit{Q}'$) \cite{Schiffler}. The following are birepresentations of $(\mathit{Q},\mathit{Q}')$.
\begin{equation} \label{diag.eq08}
\xymatrix{
&k \ar[dr]^{\begin{bmatrix}
       1            \\[0.3em]
       0         \\[0.3em]
       \end{bmatrix}} & &\\
\bar{M} &&k^2  &k \ar[l]_{\begin{bmatrix}
       1            \\[0.3em]
       1         \\[0.3em]
       \end{bmatrix}} & &k \ar[ll]_{1} \ar@/_5pc/[llllu]_{1} 
       \ar@/^5pc/[lllld]^{1} &k \ar[l]_{1}\\
&k  \ar[ur]_{\begin{bmatrix}
       0            \\[0.3em]
       1         \\[0.3em]
       \end{bmatrix}} &&
}
\end{equation}
\begin{equation} \label{diag.eq09}
\xymatrix{
&k \ar[dr]^{\begin{bmatrix}
       1            \\[0.3em]
       0         \\[0.3em]
       \end{bmatrix}} & &\\
\bar{N} &&k^2  &k \ar[l]_{\begin{bmatrix}
       1            \\[0.3em]
       1         \\[0.3em]
       \end{bmatrix}} & &k \ar[ll]_{1} \ar@/_5pc/[llllu]_{0} 
       \ar@/^5pc/[lllld]^{0} &k \ar[l]_{1}\\
&k  \ar[ur]_{\begin{bmatrix}
       0            \\[0.3em]
       1         \\[0.3em]
       \end{bmatrix}} &&
}
\end{equation}
\end{example} 

\vspace{.1cm}

\begin{definition} \cite{Abdulwahid1} \label{def.2} 
Let $\bar{V} =(V,V',\psi), \,\, \bar{W} =(W,W',\psi')$ be birepresentations of $(\mathit{Q},\mathit{Q}')$. Write 
$ V = (V_i,\phi_{\alpha})$, $V'=(V'_{i'},\mu_{\beta})$, 
$ W = (W_i,\phi'_{\alpha})$, $W'=,(W'_{i'},\mu'_{\beta})$. A morphism of birepresentations  $ \bar{f}: \bar{V} \rightarrow \bar{W}$ is a pair  $\bar{f} = (f,f')$, where $f=(f_i):  (V_i,\phi_{\alpha}) \rightarrow (W_i,\phi'_{\alpha})$, $f'=(f'_{i'}):  (V'_{i'},\mu_{\beta}) \rightarrow (W'_{i'},\mu'_{\beta})$ are morphisms in  $Rep_k(\mathit{Q})$,  $Rep_k(\mathit{Q}')$  respectively such that the following diagram commutes. 
\begin{equation} \label{diag.eq010}
\xymatrix{
V_{s(\alpha)} \ar[rr]^{\phi_{\alpha}} \ar[dr]_{f_{s(\alpha)}} && V_{t(\alpha)} \ar[dd]|\hole^(.3){\psi^{\alpha}_{\beta}} \ar[dr]^{f_{t(\alpha)}}\\
& W_{s(\alpha)} \ar[rr]|(.3){\phi'_{\alpha}} && W_{t(\alpha)} \ar[dd]^(.32){\psi'^{\alpha}_{\beta}}\\
&& V'_{s(\beta)} \ar[rr]|\hole|(.65){\mu_{\beta}} \ar[dr]_{f'_{s(\beta)}} &&
V'_{t(\beta}  \ar[dr]^{f'_{t(\beta)}} \\
&&&  W'_{s(\beta)}  \ar[rr]_{\mu'_{\beta}} &&  W'_{t(\beta)}
}
\end{equation}
The composition of two maps $(f,f')$ and $(g,g')$ can be depicted as the following diagram.
\begin{equation} \label{def.eq011}
\xymatrix{
V_{s(\alpha)} \ar[rr]^{\phi_{\alpha}} \ar[dr]_{f_{s(\alpha)}} && V_{t(\alpha)} \ar[ddd]|!{[dd];[d]}\hole|!{[ddd];[dd]}\hole^(.5){\psi^{\alpha}_{\beta}}  \ar[dr]^{f_{t(\alpha)}}\\
& W_{s(\alpha)} \ar[rr]|(.3){\phi'_{\alpha}}  \ar[dr]_{g_{s(\alpha)}}
&& W_{t(\alpha)} \ar[ddd]|!{[dd];[d]}\hole^(.5){\psi'^{\alpha}_{\beta}}  \ar[dr]^{g_{t(\alpha)}}\\
&& U_{s(\alpha)} \ar[rr]|(.3){\phi''_{\alpha}} && U_{t(\alpha)} \ar[ddd]|(.5){\psi''^{\alpha}_{\beta}} \\
&& V'_{s(\beta)} \ar[rr]|\hole|(.65){\mu_{\beta}} \ar[dr]_{f'_{s(\beta)}} &&
V'_{t(\beta}  \ar[dr]^{f'_{t(\beta)}} \\
&&&  W'_{s(\beta)}  \ar[dr]_{g'_{s(\beta)}} \ar[rr]|\hole^(.3){\mu'_{\beta}} &&  W'_{t(\beta)}\ar[dr]^{g'_{t(\beta)}}\\
&&&&  U'_{s(\beta)}  \ar[rr]_(.5){\mu''_{\beta}} &&  U'_{t(\beta)}
}
\end{equation}
In general, if $\bar{V} = (V^{(1)},V^{(2)},...,V^{(n)},\psi_{\!_{1}}, \psi_{\!_{2}}, ..., \psi_{\!_{n-1}})$, $\bar{W} = (W^{(1)},W^{(2)},...,W^{(n)},\psi'_{\!_{1}}, \psi'_{\!_{2}}, ..., \psi'_{\!_{n-1}})$ are  $n$-representations of $(\mathit{Q}_1, \,\, \mathit{Q}_2, \,\,... \,\, ,  \mathit{Q}_n)$, then a morphism of $n$-representations $\bar{f}: \bar{V} \rightarrow \bar{W}$ is $n$-tuple $\bar{f}=(f^{\!^{(1)}}, f^{\!^{(2)}}, ..., f^{\!^{(n-1)}})$, where 

\begin{center}
•$f^{\!^{(m)}} = (f^{\!^{(m)}}_{i^{(m)}}):  (V_{{i^{(m)}}},\phi^{{i^{(m)}}}_{\gamma^{(m)}}) \rightarrow  (W_{{i^{(m)}}},\mu^{{i^{(m)}}}_{\gamma^{(m)}})$,
\end{center}
 is a morphism in  $Rep_k(\mathit{Q}_m)$ for any $m \in \{2,\,\,..., \,\,n\}$, and for each pair of arrows $(\gamma^{(m-1)},\gamma^{(m)}) \in \mathrm{Q}^{(m-1)}_{\!_{1}} \times  \mathrm{Q}^{(m)}_{\!_{1}}$ the following diagram is commutative.
 
\begin{equation} \label{diag.eq010.1}
\xymatrix{ 
	V^{(m-1)}_{t^{(m-1)}(\gamma^{(m-1)})} \ar[rrr]^{\psi_{\!_{m}\gamma^{(m-1)}}^{\gamma^{(m)}}} \ar[dd]_{f^{(m-1)}_{t^{(m-1)}(\gamma^{(m-1)})}} &&& V^{(m)}_{s^{(m)}(\gamma^{(m)})} \ar[dd]^{f^{(m)}_{s^{(m)}(\gamma^{(m)})}} \\
	&&&\\
	W^{(m-1)}_{t^{(m-1)}(\gamma^{(m-1)})} \ar[rrr]_{{\psi'}_{\!_{m}\gamma^{(m-1)}}^{\gamma^{(m)}}} &&& W^{(m)}_{s^{(m)}(\gamma^{(m)})} 
}
\end{equation}
for every $m \in \{2,\,\,..., \,\,n\}$.

\vspace{.2cm}

A morphism of $n$-representations can be depicted as:

\begin{equation} \label{diag.eq010.2}
\xymatrix{ 
	&  \ar[rr]&&  \ar[dd]&& \ar[rr]\ar[dd]|\hole &&  \ar[dd] &&  \ar[rr] \ar[dd] &&  \ar[dd] &\\
	 \ar[ur] \ar[rr]    &&  \ar[dd] \ar[ur]&& \ar[rr]\ar[dd] \ar[ur]&&  \ar[dd] \ar[ur]&&   \ar[rr] \ar[ur] \ar[dd]&&   \ar[dd] \ar[ur]&\\
	&&& \ar[rr]|\hole&&   &&  \ar[rr]|\hole&&   && \ar@{--}[r]  &\\
	&& \ar[rr]\ar[ur]&&  \ar[ur]&&  \ar[rr]\ar[ur]&&  \ar[ur]  & &   \ar[ur] \ar@{--}[rr] &&
}
\end{equation}

\end{definition}

\vspace{.2cm}

\begin{remark} \cite{Abdulwahid1} \label{r.1} \textbf{•}

\begin{enumerate}[label=(\roman*)]
\item The above definition gives rise to form a category  $Rep_{\!_{(\mathit{Q},\mathit{Q}')}}$ of $k$-linear birepresentations of  $(\mathit{Q},\mathit{Q}')$. We denote by  $rep_{\!_{(\mathit{Q},\mathit{Q}')}}$ the full subcategory of $Rep_{\!_{(\mathit{Q},\mathit{Q}')}}$  consisting of the finite dimensional birepresentations. Similarly, it also creates a category $Rep_{\!_{(\mathit{Q}_1,\mathit{Q}_2,...,\mathit{Q}_n)}}$ of $n$-representations. We denote $rep_{\!_{(\mathit{Q}_1,\mathit{Q}_2,...,\mathit{Q}_n)}}$ the full subcategory of $Rep_{\!_{(\mathit{Q}_1,\mathit{Q}_2,...,\mathit{Q}_n)}}$ consisting of the finite dimensional  $n$-representations.

\item For any  $m \in \{2,\,\,..., \,\,n\}$, let $\mathit{Q}_m = (\mathrm{Q}^{(m)}_{\!_{0}}, \mathrm{Q}^{(m)}_{\!_{1}}, s^{(m)},t^{(m)})$ be a quiver and fix $j \in \{2,\,\,..., \,\,n\}$. Let $\Upsilon_{Rep_k(\mathit{Q}_j)}$ be the subcategory of $Rep_{\!_{(\mathit{Q}_1,\mathit{Q}_2,...,\mathit{Q}_n)}}$ whose objects are $(2n-1)$-tuples  $\bar{X} = (0,0,...,V^{(j)},0,...,0,\psi_{\!_{1}}, \psi_{\!_{2}}, ..., \psi_{\!_{n-1}})$, where $V^{(j)}$ is a representation of $\mathit{Q}_j$,  and $\psi_{\!_{m}\gamma^{(m-1)}}^{\gamma^{(m)}} = 0$ for every pair of arrows $(\gamma^{(m-1)},\gamma^{(m)}) \in \mathrm{Q}^{(m-1)}_{\!_{1}} \times  \mathrm{Q}^{(m)}_{\!_{1}}$ and  $m \in \{2,\,\,..., \,\,n\}$. Then $\Upsilon_{Rep_k(\mathit{Q}_j)}$ is clearly a full subcategory of  $Rep_{\!_{(\mathit{Q}_1,\mathit{Q}_2,...,\mathit{Q}_n)}}$. Notably, we have an equivalence of categories $\Upsilon_{Rep_k(\mathit{Q}_j)} \simeq Rep_k(\mathit{Q}_j)$, and thus by by Proposition (\ref{p.Equiv.Rep.Mod}), we have $Rep_k(\mathit{Q}_j) \simeq \Upsilon_{Rep_k(\mathit{Q}_j)} \simeq Mod \, k\mathit{Q}_j $. It turns out that the category  $Rep_k(\mathit{Q}_j)$ and $ Mod \, k\mathit{Q}_j)$ can be identified as full subcategories of $Rep_{\!_{(\mathit{Q}_1,\mathit{Q}_2,...,\mathit{Q}_n)}}$. \\
The category $\Upsilon_{Rep_k(\mathit{Q}_j)}$ has a full subcategory  $\Upsilon_{rep_k(\mathit{Q}_j)}$ when we restrict the objects on the finite dimensional representations. Therefore, we also have  $rep_k(\mathit{Q}_j) \simeq \Upsilon_{rep_k(\mathit{Q}_j)} \simeq mod \, k\mathit{Q}_j $. \\
\end{enumerate}
\end{remark}

\begin{remark} \label{r.bicategory}  \,\, Let $\Bb_0$ be the class of all quivers. One might consider the class $\Bb_0$ and full subcategories of the categories of birepresentations of quivers to build a bicategory. Indeed, there is a bicategory $\Bb$ consists of 
\begin{itemize}
\item the objects or the $0$-cells of $\Bb$ are simply the elements of  $\Bb_0$
 
\item for each $Q, Q' \in \Bb_0$, we have  $\Bb(Q,Q') =  Rep_k(Q) \times Rep_k(Q')$, whose objects are the $1$-cells of $\Bb$, and whose   morphisms are the $2$-cells of $\Bb$

\item for each $Q, Q', Q'' \in \Bb_0$, \,\, a composition functor 
\begin{center}
•$\Ff: Rep_k(Q') \times Rep_k(Q'') \,\,\, \times \,\,\, Rep_k(Q) \times Rep_k(Q') \rightarrow Rep_k(Q) \times Rep_k(Q'')$
\end{center}
defined by:\\
\begin{center}
•$\Ff ((N',N''),(M,M')) = (M,N''), \,\,\,\,\, \Ff ((g',g''),(f,f')) = (f,g'')$
on $1$-cells $(M,M'), (N',N'')$ and $2$-cells $(f,f'), (g',g'')$.
\end{center}
\item for any  $Q \in \Bb_0$ and for each $(M,M') \in \Bb(Q,Q)$, we have $\Ff ((M,M'),(M,M)) = (M,M')$ \,\, and \,\, $\Ff ((M',M'),(M,M')) = (M,M')$. Furthermore, for any $2$-cell $(f,f')$, we have $\Ff ((f',f'),(f,f')) = (f,f')$  \,\, and \,\, $\Ff ((f,f'),(f,f)) = (f,f')$. Thus, the identity and the unit coherence axioms hold. 
\end{itemize}

The rest of bicategories axioms are obviously satisfied. \,\, For each $Q, Q' \in \Bb_0$, let $\beth_{\!_{(Q,Q')}}$ be the full subcategory of $Rep_{\!_{(Q,Q')}}$ whose objects are the triples $(X,X',\Psi)$, where $(X,X') \in Rep_k(Q) \times Rep_k(Q')$ and $\Psi^{\alpha}_{\beta} = 0$ for every pair of arrows $(\alpha,\beta) \in \mathsf{Q}_{\!_{1}} \times  \mathsf{Q}'_{\!_{1}}$, and whose morphisms are usual morphisms of birepresentations between them. \,\, Clearly, \,\, $\beth_{\!_{(Q,Q')}} \cong Rep_k(Q) \times Rep_k(Q')$  for any $Q, Q' \in \Bb_0$. Thus,  by considering the class $\Bb_0$ and the full subcategories described above of the birepresentations categories of quivers, we can always build a bicategory as above.\\
Obviously, \,\, the discussion above implies that for each $Q, Q' \in \Bb_0$, the product category  $Rep_k(Q) \times Rep_k(Q')$ can be viewed as a full subcategory of $Rep_{\!_{(Q,Q')}}$. Further, it implies that the product category $Rep_k(Q_1) \times Rep_k(Q_2) \times ...\times  Rep_k(Q_n) $ can be viewed as a full subcategory of $Rep_{\!_{(Q_1,Q_2,...,Q_n)}}$, where $Q_1,Q_2,...,Q_n \in \Bb_0$ and $n \geq 2$.\\
We also have the same analogue if we replace $Rep_{\!_{(Q_1,Q_2,...,Q_n)}}$,  by $rep_{\!_{(Q_1,Q_2,...,Q_n)}}$, and $Rep_k(Q_1)$, $Rep_k(Q_2)$, ... ,  $Rep_k(Q_n) $ by  $rep_k(Q_1), rep_k(Q_2)$,  ... , $rep_k(Q_n) $  respectively. For the basic notions of bicategories, we refer the reader to \cite{Leinster2}. \\
\end{remark}

\begin{definition} \cite{Abdulwahid1} \label{def.3} 
Let $\bar{V} = (V,V',\psi)$, $\bar{W} = (W,W',\psi')$ be birepresentations of $(\mathit{Q},\mathit{Q}')$. 
Write $ V = (V_i,\phi_{\alpha})$, $V' = (V'_{i'},\phi'_{\beta})$,  
$W = ((W_i,\mu_{\alpha})$, $W' = (W'_{i'},\mu'_{\beta})$. Then 
\begin{equation} \label{diag.eq018}
\bar{V} \oplus \bar{W} \,\,\ = \,\,\ ((V_i \oplus W_i,\begin{bmatrix}
       \phi_{\alpha} & 0            \\[0.3em]
       0 & \mu_{\alpha}         \\[0.3em] \end{bmatrix}), (V'_{i'} \oplus W'_{i'},  \begin{bmatrix}
       \phi'_{\alpha} & 0            \\[0.3em]
       0 & \mu'_{\alpha}         \\[0.3em] \end{bmatrix})
, \begin{bmatrix}
       \psi^{\alpha}_{\beta} & 0            \\[0.3em]
       0 & \psi'^{\alpha}_{\beta}         \\[0.3em] \end{bmatrix}),
\end{equation}
where $(V_i \oplus W_i,\begin{bmatrix}
       \phi_{\alpha} & 0            \\[0.3em]
       0 & \mu_{\alpha}         \\[0.3em] \end{bmatrix})$, $(V'_{i'} \oplus W'_{i'},  \begin{bmatrix}   \phi'_{\alpha} & 0   \\[0.3em]
       0 & \mu'_{\alpha}   \\[0.3em] \end{bmatrix})$ are the direct sums of $(V_i,\phi_{\alpha}),\, (W_i,\mu_{\alpha})$ and $(V'_{i'}, \phi'_{\beta}) ,\, (W'_{i'}, \mu'_{\beta})$ in $Rep_k(\mathit{Q})$, $Rep_k(\mathit{Q}')$ respectively, is a  birepresentation of $(\mathit{Q},\mathit{Q}')$ called the \textbf{direct sum} of  $\bar{V}, \,\  \bar{W}$ (in $Rep_{\!_{(\mathit{Q},\mathit{Q}')}}$). \\

Similarly, direct sums in $Rep_{\!_{(\mathit{Q}_1,\mathit{Q}_2,...,\mathit{Q}_n)}}$ can be defined. \\
\end{definition}

\begin{definition} \cite{Abdulwahid1} \label{def.4} 
A birepresentation $\bar{V} \in Rep_{\!_{(\mathit{Q},\mathit{Q}')}}$ is called \textbf{indecomposable} if $\bar{M} \neq 0$ and $\bar{M}$ cannot be written as a direct sum of two nonzero birepresentations, that is,
whenever $\bar{M} \cong \bar{L} \oplus \bar{N}$ with $ \bar{L}, \bar{N} \in Rep_{\!_{(\mathit{Q},\mathit{Q}')}}$, then $\bar{L} = 0 $ or $\bar{N} = 0$.

\end{definition}

\vspace{.2cm}

We end the section with the following consequence.\\

\begin{theorem} \cite{Abdulwahid1} \label{thm.Abelian2} 
The category  $Rep_{\!_{(\mathit{Q}_1,\mathit{Q}_2,...,\mathit{Q}_{n})}}$ is a $k$-linear abelian category for any integer $n \geq 2$. \\
\end{theorem} 

 \vspace{.2cm}

For more concrete examples and proofs of the theorems above, we refer the reader to \cite{Abdulwahid1}.\\

 \vspace{.2cm}

 \section{\textbf{Completeness and Cocompleteness in The Categories of $n$-representations}}\label{s.Completeness.Cocompleteness}

Recall that a  category $\mathfrak{C}$ is cocomplete when every functor $\mathfrak{F}: \mathfrak{D} \rightarrow \mathfrak{C}$,
with $\mathfrak{D}$ a small category has a colimit \cite{Borceux1}. For the basic notions of cocomplete categories and examples, we refer to  \cite{Adamek}, \cite{Borceux1}, or \cite{Schubert}. A  functor is \textit{cocontinuous} if it preserves all small colimits \cite[p. 142]{Freyd}.

\begin{proposition} \label{p.7}  
The category $Rep_{\!_{(\mathit{Q},\mathit{Q}')}}$ is complete and the forgetful functor 
\begin{center}
•${\mathcal{U}}: Rep_{\!_{(\mathit{Q},\mathit{Q}')}}  \rightarrow  Rep_k(\mathit{Q}) \times Rep_k(\mathit{Q}')$
\end{center}
 is continuous. Furthermore, the limit of objects in $Rep_{\!_{(\mathit{Q},\mathit{Q}')}}$ can be obtained by the corresponding construction for objects in $Rep_k(\mathit{Q}) \times Rep_k(\mathit{Q}')$.\\  
\end{proposition} 

\begin{proof}

Let $\mathscr{D}$ be a small category, and let $\mathscr{F}:\mathscr{D} \rightarrow Rep_{\!_{(\mathit{Q},\mathit{Q}')}}$ be a functor, and  consider the composition of the following functors.
\begin{equation} \label{diag.eq21}
\xymatrix{
\mathscr{D} \ar[r]^(.4){\mathscr{F}} & Rep_{\!_{(\mathit{Q},\mathit{Q}')}}  \ar[r]^(.36){\mathcal{U}} &  Rep_k(\mathit{Q}) \times Rep_k(\mathit{Q}') \ar[r]^(.6){\mathcal{P}_1} &  Rep_k(\mathit{Q}) 
}
\end{equation} 

\begin{equation} \label{diag.eq22}
\xymatrix{
\mathscr{D} \ar[r]^(.4){\mathscr{F}} & Rep_{\!_{(\mathit{Q},\mathit{Q}')}}  \ar[r]^(.36){\mathcal{U}'} &  Rep_k(\mathit{Q}) \times Rep_k(\mathit{Q}') \ar[r]^(.6){\mathcal{P}_2} &  Rep_k(\mathit{Q}')
}
\end{equation} 

where  $\mathcal{U}$,  $\mathcal{U}'$ are the obvious forgetful functors, and $\mathcal{P}_1$, $\mathcal{P}_2$ are the projection functors. \\
For all $D \in \mathscr{D}$, let $\mathscr{F}D = \bar{V} = (V,V',\psi^{D,})$. Since $ Rep_k(\mathit{Q})$, $ Rep_k(\mathit{Q}')$ are complete categories, the functors $\mathcal{P}_1 \mathcal{U} \mathscr{F}$, $\mathcal{P}_2 \mathcal{U} \mathscr{F}$ have limits. Let 
 $(L,(\eta^{\!_{D}})_{\!_{{D} \in \mathscr{D}}})$, 
  $(L',(\eta'^{\!_{D}})_{\!_{{D} \in \mathscr{D}}})$ be limits of 
  $\mathcal{P}_1 \mathcal{U} \mathscr{F}$, $\mathcal{P}_2 \mathcal{U} \mathscr{F}$ respectively, where $\eta^{\!_{D}}, \,\ \eta'^{\!_{D}}$ are the morphisms

 \begin{equation}\label{diag.eq23}
L  \xrightarrow{\eta^{\!_{D}}} V^D, \,\,\, \,\,\,\,\,\, \,\,\, \,\,\,\,\,\, \,\,\, \,\,\,\,\,\, \,\,\, \,\,\,\,\,\  L'  \xrightarrow{\eta'^{\!_{D}}} V'^D \,\,\ ,
\end{equation}
for every $D \in \mathscr{D}$. Write 
$L=(L_i,\phi'_{\alpha})$, $L'=(L'_{i'},\mu'_{\beta})$, 
$V^D=(V^D_i,\phi^D_{\alpha})$, 
$V^{D'}=(V^{D'}_i,\phi^{D'}_{\alpha})$
$V'^{D'} = (V'^{D'}_{i'},\mu^{D'}_{\beta})$,
$V'^D=(V'^D_i,\mu^D_{\beta})$. \\

Let $ h:D \rightarrow D'$ be a morphism in $\mathscr{D}$ and consider the following diagram

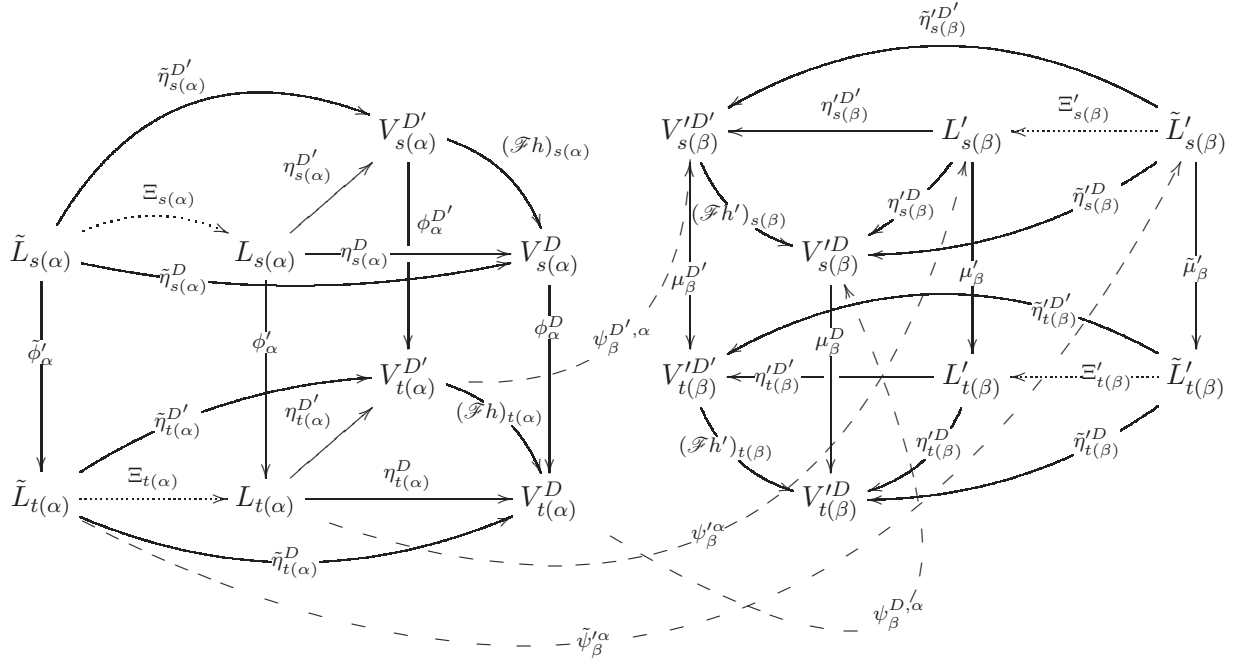
\begin{equation} \label{diag.eq24}
\xymatrix{ 
	&&& V^{D'}_{s(\alpha)} \ar@/^1pc/[dr]^{(\mathscr{F}h)_{s(\alpha)}}  \ar[dd]|\hole^(.36){\phi^{D'}_{\alpha}}&&V'^{D'}_{s(\beta)} \ar@/_1pc/[dr]|(.5){(\mathscr{F}h')_{s(\beta)}} \ar[dd]|(.6){\mu^{D'}_{\beta}}&& L'_{s(\beta)} \ar[dd]|(.58){\mu'_{\beta}}\ar@/^.5pc/[dl]|(.5){{\eta'^{D}_{s(\beta)}}}   \ar[ll]_(.45){{\eta'^{D'}_{s(\beta)}}}   &&
\tilde{L}'_{s(\beta)}   \ar@{..>}[ll]_(.5){{\Xi'_{s(\beta)}}}  \ar[dd]|(.56){\tilde{\mu}'_{\beta}}
 \ar@/^1.2pc/[dlll]|(.3){{\tilde{\eta}'^{D}_{s(\beta)}}} \ar@/_2.7pc/[llll]_(.5){{\tilde{\eta}'^{D'}_{s(\beta)}}}	
	\\
\tilde{L}_{s(\alpha)} \ar@/^1.2pc/@{..>}[rr]^(.56){{\Xi_{s(\alpha)}}}  
\ar[dd]|(.4){\tilde{\phi}'_{\alpha}} \ar@/_1pc/[rrrr]|(.28){{\tilde{\eta}^{D}_{s(\alpha)}}} \ar@/^3pc/[urrr]^{{\tilde{\eta}^{D'}_{s(\alpha)}}}	
	&&L_{s(\alpha)}  \ar[dd]|(.36){\phi'_{\alpha}} \ar[rr]|(.35){{\eta^{D}_{s(\alpha)}}} \ar[ur]^{{\eta^{D'}_{s(\alpha)}}}&& V^{D}_{s(\alpha)} \ar[dd]|(.3){\phi^{D}_{\alpha}}&& V'^{D}_{s(\beta)}  \ar[dd]|(.36){\mu^{D}_{\beta}}&&&
 	\\
	&&& V^{D'}_{t(\alpha)}  \ar@/_3.5pc/@{-->}[uurr]|(.5){\psi^{D',\alpha}_{\beta}}
	\ar@/^1pc/[dr]|(.5){(\mathscr{F}h)_{t(\alpha)}} && V'^{D'}_{t(\beta)} \ar@/_1pc/[dr]|(.4){(\mathscr{F}h')_{t(\beta)}} && L'_{t(\beta)} \ar[ll]|\hole|(.69){{\eta'^{D'}_{t(\beta)}}} \ar@/^1pc/[dl]|(.37){{\eta'^{D}_{t(\beta)}}}&& 
	\tilde{L}'_{t(\beta)}  \ar@{..>}[ll]|(.4){{\Xi'_{t(\beta)}}}    
		 \ar@/^1.2pc/[dlll]|(.3){{\tilde{\eta}'^{D}_{t(\beta)}}} \ar@/_2.6pc/[llll]|(.28){{\tilde{\eta}'^{D'}_{t(\beta)}}}	\\
\tilde{L}_{t(\alpha)}  \ar@{..>}[rr]^(.5){{\Xi_{t(\alpha)}}} \ar@/_10pc/@{-->}[uuurrrrrrrrr]|(.4){\tilde{\psi}'^{\alpha}_{\beta}} \ar@/_2pc/[rrrr]|(.5){{\tilde{\eta}^{D}_{t(\alpha)}}} \ar@/^1pc/[urrr]|(.4){{\tilde{\eta}^{D'}_{t(\alpha)}}}			&& L_{t(\alpha)} \ar@/_7.5pc/@{-->}[uuurrrrr]|(.47){\psi'^{\alpha}_{\beta}}
\ar[ur]^{\eta^{D'}_{t(\alpha)}} \ar[rr]^{\eta^{D}_{t(\alpha)}} && V^{D}_{t(\alpha)} \ar@/_10pc/@{-->}[uurr]|(.5){\psi^{D,\alpha}_{\beta}} &&  V'^{D}_{t(\beta)} &&
}
\end{equation}

\vspace{.3cm}

Clearly,  $(L'_{s(\beta)},(\eta'^{\!_{D}}_{s(\beta)})_{\!_{{D} \in \mathscr{D}}})$  can be viewed as a limit of  a functor $\rho: \mathscr{D} \rightarrow Vect_k$. Furthermore, we have for any $D, \,\, D'  \in \mathscr{D}$,\\

\begin{tabular}{lllll}
$(\mathscr{F}h')_{s(\beta)} \,\, \psi^{D',\alpha}_{\beta} \,\, \eta^{D'}_{t(\alpha)} $  &  $=\psi^{D,\alpha}_{\beta} \,\, (\mathscr{F}h)_{t(\alpha)} \,\, \eta^{D'}_{t(\alpha)} $\\
&  (since $\bar{\mathscr{F}h} = (\mathscr{F}h, \mathscr{F}h')$ is a morphism in $Rep_{\!_{(\mathit{Q},\mathit{Q}')}})$\\
 & $=\psi^{D,\alpha}_{\beta} \,\, \eta^{D}_{t(\alpha)} $\\
&  (since $(L,(\eta^{\!_{D}})_{\!_{{D} \in \mathscr{D}}})$ is a limit of  $\mathcal{P}_1 \mathcal{U} \mathscr{F}$) \\
 \end{tabular}\\
 
\vspace{.2cm}

Thus, $(L_{t(\alpha)}, (\psi^{D,} \,\, \eta^{D}_{t(\alpha)})_{\!_{{D} \in \mathscr{D}}}) $ is a cone on  $\rho$. Since $Vec_k$ is complete,  there exists a unique $k$-linear map $\psi'^{\alpha}_{\beta}: L_{t(\alpha)} \rightarrow L'_{s(\beta)}$ with $\eta'^{\!_{D}}_{s(\beta)} \,\, \psi'^{\alpha}_{\beta} = \psi^{D,\alpha}_{\beta} \,\, \eta^{D}_{t(\alpha)}$ for every $D  \in \mathscr{D}$ and for each pair of arrows $(\alpha,\beta) \in \mathrm{Q}_{\!_{1}} \times \in \mathit{Q}'_{\!_{1}}$.  Hence,  $\bar{\eta}^{D} $ is a morphism in $Rep_{\!_{(\mathit{Q},\mathit{Q}')}}$ for any $D  \in \mathscr{D}$.
Let $\bar{L} =(L,L',\psi')$, $\bar{\eta}^{D} = (\eta^{D},\eta'^{D})$. We claim that $(\bar{L},(\bar{\eta}^{\!_{D}})_{\!_{{D} \in \mathscr{D}}})$ is a limit of $\mathscr{F}$. Then obviously all we need is to show that for any cone $ (\bar{\tilde{L}}, (\bar{\tilde{\eta}}^{\!_{D}})_{\!_{{D} \in \mathscr{D}}})$, there exists a unique morphism $\bar{\Xi} = (\Xi,\Xi')$ in $Rep_{\!_{(\mathit{Q},\mathit{Q}')}}$ with $\bar{\eta}^D \,\, \bar{\Xi} = \bar{\tilde{\eta}}^D$ for every $D  \in \mathscr{D}$. Let $ (\bar{\tilde{L}}, (\bar{\tilde{\eta}}^{\!_{D}})_{\!_{{D} \in \mathscr{D}}})$ be a cone on  $\mathscr{F}$ and write $\bar{\tilde{L}} = (\tilde{L},\tilde{L}',\tilde{\psi}')$. Since $Rep_k(\mathit{Q}) ,\,\ Rep_k(\mathit{Q}')$ are complete categories, exists unique morphisms $ \Xi, \,\, \Xi'$ in $Rep_k(\mathit{Q}) ,\,\ Rep_k(\mathit{Q}')$ respectively such that $\eta^D \,\, \Xi = \tilde{\eta}^D$, $\eta'^D \,\, \Xi' = \tilde{\eta}'^D$ for every $D  \in \mathscr{D}$. It remains to show that $\bar{\Xi} = (\Xi,\Xi')$ is a morphism in $Rep_{\!_{(\mathit{Q},\mathit{Q}')}}$. \\
For any $D  \in \mathscr{D}$, we have\\

\begin{tabular}{lllll}
$\eta'^{D}_{s(\beta)} \,\, \Xi'_{s(\beta)} \,\, \tilde{\psi}'^{\alpha}_{\beta} $  &  $ = \tilde{\eta}'^{D}_{s(\beta)} \,\ \tilde{\psi}'^{\alpha}_{\beta} $ \\
&  (since $\eta'^D \,\, \Xi' = \tilde{\eta}'^D$) \\
&  $= \psi^{D,\alpha}_{\beta} \,\, \tilde{\eta}^{D}_{t(\alpha)} $\\
&  (since $\bar{\tilde{\eta}}'^{D}$ is a morphism in $Rep_{\!_{(\mathit{Q},\mathit{Q}')}}$)\\
 & $=\psi^{D,\alpha}_{\beta} \,\, \eta^{D}_{t(\alpha)} \,\, \Xi_{t(\alpha)} $\\
&  (since $\eta^D \,\, \Xi = \tilde{\eta}^D$) \\
& $= \eta'^{D}_{s(\beta)} \psi'^{\alpha}_{\beta} \,\, \Xi_{t(\alpha)} $\\
&  (since $\bar{\eta}^{D} $ is a morphism in $Rep_{\!_{(\mathit{Q},\mathit{Q}')}}$)\\
 \end{tabular}\\
 
\vspace{.2cm}

Notably,  $(\tilde{L}_{t(\alpha)}, (\tilde{\eta}'^{D}_{s(\beta)} \, \tilde{\psi}')_{\!_{{D} \in \mathscr{D}}}) $ can be viewed as a limit of  a functor $\rho': \mathscr{D} \rightarrow Vect_k$. Furthermore, $(\tilde{L}_{t(\alpha)}, (\tilde{\eta}'^{D}_{s(\beta)} \, \tilde{\psi}'^{\alpha}_{\beta})_{\!_{{D} \in \mathscr{D}}}) $ is a cone on $\rho'$ since\\

\begin{tabular}{lllll}
$(\mathscr{F}h')_{s(\beta)} \,\, \tilde{\eta}'^{D'}_{s(\beta)} \,\, \tilde{\psi}'^{\alpha}_{\beta}$  &  $ = \mathscr{F}h'_{s(\beta)}  \,\, \eta'^{D'}_{s(\beta)} \,\, \Xi'_{s(\beta)}  \,\, \tilde{\psi}'^{\alpha}_{\beta}$ \\
&  (since $\bar{\tilde{\eta}}^{D'} = \bar{\eta}^{D'} \,\, \bar{\Xi}$)\\
 & $ = \eta'^{D}_{s(\beta)} \,\, \Xi'_{s(\beta)}  \,\, \tilde{\psi}'^{\alpha}_{\beta}$ \\
&  (since $\bar{\mathscr{F}h} \,\, \bar{\eta}^{D'} = \bar{\eta}^{D}$) \\
 & $ = \tilde{\eta}'^{D}_{s(\beta)} \,\, \tilde{\psi}'^{\alpha}_{\beta} \,\,\,\,\,\,$. \\
 \end{tabular}\\
 
\vspace{.2cm}

Since $Vec_k$ is complete,  it follows from the universal property of the limit that $\Xi'_{s(\beta)} \,\, \tilde{\psi}'^{\alpha}_{\beta} = \psi'^{\alpha}_{\beta} \,\, \Xi_{t(\alpha)}$. Thus,  $\bar{\Xi}$ is a morphism in $Rep_{\!_{(\mathit{Q},\mathit{Q}')}}$. \,\, Thus,  $(\bar{L},(\bar{\eta}^{\!_{D}})_{\!_{{D} \in \mathscr{D}}})$ is a limit of $\mathscr{F}$, as desired.\\

\end{proof}

From Proposition (\ref{p.7}) and Remark (\ref{r.inductive}), we obtain.\\

\begin{proposition} \label{p.7.1}  
The category $Rep_{\!_{(\mathit{Q}_1,\mathit{Q}_2,...,\mathit{Q}_n)}}$ is complete and the forgetful functor 
\begin{center}
•$\mathcal{U}_n: Rep_{\!_{(\mathit{Q}_1,\mathit{Q}_2,...,\mathit{Q}_n)}}  \rightarrow  Rep_k(\mathit{Q}_1) \times Rep_k(\mathit{Q}_2) \times . . . \times Rep_k(\mathit{Q}_n)$
\end{center}
 is continuous. Furthermore, the limit of objects in $Rep_{\!_{(\mathit{Q}_1,\mathit{Q}_2,...,\mathit{Q}_n)}}$ can be obtained by the corresponding construction for objects in $Rep_k(\mathit{Q}_1) \times Rep_k(\mathit{Q}_2) \times . . . \times Rep_k(\mathit{Q}_n)$.\\  
\end{proposition} 

\begin{proof}

\end{proof}

\begin{proposition} \label{p.8}  
The category $Rep_{\!_{(\mathit{Q},\mathit{Q}')}}$ is cocomplete and the forgetful functor 
\begin{center}
•${\mathcal{U}}: Rep_{\!_{(\mathit{Q},\mathit{Q}')}}  \rightarrow  Rep_k(\mathit{Q}) \times Rep_k(\mathit{Q}')$
\end{center}
 is cocontinuous. Moreover, the colimit of objects in $Rep_{\!_{(\mathit{Q},\mathit{Q}')}}$ can be obtained by the corresponding construction for objects in $Rep_k(\mathit{Q}) \times Rep_k(\mathit{Q}')$.\\  
\end{proposition} 

\begin{proof}

Let $\mathscr{D}$ be a small category, and let $\mathscr{F}:\mathscr{D} \rightarrow Rep_{\!_{(\mathit{Q},\mathit{Q}')}}$ be a functor, and  consider the composition of the following functors.
\begin{equation} \label{diag.eq25}
\xymatrix{
\mathscr{D} \ar[r]^(.4){\mathscr{F}} & Rep_{\!_{(\mathit{Q},\mathit{Q}')}}  \ar[r]^(.36){\mathcal{U}} &  Rep_k(\mathit{Q}) \times Rep_k(\mathit{Q}') \ar[r]^(.6){\mathcal{P}_1} &  Rep_k(\mathit{Q}) 
}
\end{equation} 

\begin{equation} \label{diag.eq26}
\xymatrix{
\mathscr{D} \ar[r]^(.4){\mathscr{F}} & Rep_{\!_{(\mathit{Q},\mathit{Q}')}}  \ar[r]^(.36){\mathcal{U}'} &  Rep_k(\mathit{Q}) \times Rep_k(\mathit{Q}') \ar[r]^(.6){\mathcal{P}_2} &  Rep_k(\mathit{Q}')
}
\end{equation} 

where  $\mathcal{U}$,  $\mathcal{U}'$ are the obvious forgetful functors, and $\mathcal{P}_1$, $\mathcal{P}_2$ are the projection functors. \\
For all $D \in \mathscr{D}$, let $\mathscr{F}D = \bar{V} = (V,V',\psi^{D,})$. Since $ Rep_k(\mathit{Q})$, $ Rep_k(\mathit{Q}')$ are cocomplete categories, the functors $\mathcal{P}_1 \mathcal{U} \mathscr{F}$, $\mathcal{P}_2 \mathcal{U} \mathscr{F}$ have colimits. Let 
 $(C,(\zeta^{\!_{D}})_{\!_{{D} \in \mathscr{D}}})$, 
  $(C',(\zeta'^{\!_{D}})_{\!_{{D} \in \mathscr{D}}})$ be colimits of 
  $\mathcal{P}_1 \mathcal{U} \mathscr{F}$, $\mathcal{P}_2 \mathcal{U} \mathscr{F}$ respectively, where $\zeta^{\!_{D}}, \,\ \zeta'^{\!_{D}}$ are the morphisms

 \begin{equation}\label{diag.eq27}
V^D  \xrightarrow{\zeta^{\!_{D}}} C, \,\,\, \,\,\,\,\,\, \,\,\, \,\,\,\,\,\, \,\,\, \,\,\,\,\,\, \,\,\, \,\,\,\,\,\  V'^D \xrightarrow{\zeta'^{\!_{D}}} C'  \,\,\ ,
\end{equation}
for every $D \in \mathscr{D}$. Write 
$C=(C_i,\phi'_{\alpha})$, $C'=(C'_{i'},\mu'_{\beta})$, 
$V^D=(V^D_i,\phi^D_{\alpha})$, 
$V^{D'}=(V^{D'}_i,\phi^{D'}_{\alpha})$
$V'^{D'} = (V'^{D'}_{i'},\mu^{D'}_{\beta})$,
$V'^D=(V'^D_i,\mu^D_{\beta})$. \\

Let $ h:D \rightarrow D'$ be a morphism in $\mathscr{D}$ and consider the following diagram.\\

\begin{equation} \label{diag.eq28}
\xymatrix{ 
	&&& V^D_{s(\alpha)}
	\ar@/_3pc/[dlll]_{{\tilde{\zeta}^{D}_{s(\alpha)}}}
	 \ar[dl]_(.5){{\zeta^{D}_{s(\alpha)}}}
	\ar@/^1pc/[dr]^(.4){(\mathscr{F}h)_{s(\alpha)}}
	 \ar[dd]|\hole^(.3){\phi^{D}_{\alpha}}&&V'^D_{s(\beta)}
	 \ar@/^2.7pc/[rrrr]^(.5){{\tilde{\zeta}'^{D}_{s(\beta)}}}
	 \ar[dl]_(.5){{\zeta^{D}_{s(\alpha)}}}
	\ar[rr]^(.45){{\zeta'^{D}_{s(\beta)}}}  
	\ar@/_1pc/[dr]|(.5){(\mathscr{F}h')_{s(\beta)}} \ar[dd]|(.6){\mu^{D}_{\beta}}&& C'_{s(\beta)} \ar@{..>}[rr]^(.5){{\Lambda'_{s(\beta)}}}
 \ar[dd]|(.58){\mu'_{\beta}}    &&
\tilde{C}'_{s(\beta)}     \ar[dd]|(.56){\tilde{\mu}'_{\beta}} 	\\
\tilde{C}_{s(\alpha)}   
\ar[dd]|(.4){\tilde{\phi}'_{\alpha}} 	
&& C_{s(\alpha)} \ar@/_1.2pc/@{..>}[ll]_(.46){{\Lambda_{s(\alpha)}}}
	 \ar[dd]|(.36){\phi'_{\alpha}}  && V^{D'}_{s(\alpha)}  \ar@/^1.4pc/[llll]|(.69){{\tilde{\zeta}^{D'}_{s(\alpha)}}}
	\ar[ll]|(.35){{\zeta^{D'}_{s(\alpha)}}}
	 \ar[dd]|(.3){\phi^{D'}_{\alpha}}&& V'^{D'}_{s(\beta)}
	 \ar@/_1.2pc/[urrr]|(.7){{\tilde{\zeta}'^{D'}_{s(\beta)}}}  
	\ar@/_.5pc/[ur]|(.5){{\zeta'^{D'}_{s(\beta)}}}
	\ar[dd]|(.36){\mu^{D'}_{\beta}}&&&
 	\\
	&&& V^D_{t(\alpha)} 
	\ar[dl]_{\zeta^{D}_{t(\alpha)}}
	 \ar@/_1pc/[dlll]|(.58){{\tilde{\zeta}^{D}_{t(\alpha)}}}
	  \ar@/_3.5pc/@{-->}[uurr]|(.5){\psi^{D,\alpha}_{\beta}}
	\ar@/^1pc/[dr]|(.5){(\mathscr{F}h)_{t(\alpha)}} && V'^D_{t(\beta)} 
		\ar[rr]|\hole _(.36){{\zeta'^{D}_{t(\beta)}}}
	\ar@/_1pc/[dr]|(.4){(\mathscr{F}h')_{t(\beta)}} && C'_{t(\beta)}  
	 \ar@{..>}[rr]|(.6){{\Lambda'_{t(\beta)}}}    && 
	\tilde{C}'_{t(\beta)} 
	 \ar@/_2.6pc/[llll]|(.28){{\tilde{\zeta}'^{D}_{t(\beta)}}}	\\
\tilde{C}_{t(\alpha)} 
  \ar@/_10pc/@{-->}[uuurrrrrrrrr]|(.4){\tilde{\psi}'^{\alpha}_{\beta}} 
 && C_{t(\alpha)} 
\ar@{..>}[ll]_(.5){{\Lambda_{t(\alpha)}}} 
 \ar@/_7.5pc/@{-->}[uuurrrrr]|(.47){\psi'^{\alpha}_{\beta}}
&& V^{D'}_{t(\alpha)} 
 \ar[ll]_{\zeta^{D'}_{t(\alpha)}} 
\ar@/^2pc/[llll]|(.5){{\tilde{\zeta}^{D'}_{t(\alpha)}}}
 \ar@/_10pc/@{-->}[uurr]|(.5){\psi^{D',\alpha}_{\beta}} &&  V'^{D'}_{t(\beta)} 
\ar@/_1.2pc/[urrr]|(.66){{\tilde{\zeta}'^{D'}_{t(\beta)}}}
\ar@/_1pc/[ur]|(.6){{\zeta'^{D'}_{t(\beta)}}} &&
}
\end{equation}

\vspace{.3cm}

It is clear that  $(C_{t(\alpha)},(\zeta^{\!_{D}}_{t(\alpha)})_{\!_{{D} \in \mathscr{D}}})$  can be viewed as a colimit of  a functor $\upsilon: \mathscr{D} \rightarrow Vect_k$. Moreover, we have for any $D, \,\, D'  \in \mathscr{D}$,\\

\begin{tabular}{lllll}
$\zeta'^{\!_{D'}}_{s(\beta)} \,\, \psi^{D',\alpha}_{\beta} \,\, 
(\mathscr{F}h')_{t(\alpha)}$  &  $=\zeta'^{\!_{D'}}_{s(\beta)} \,\, (\mathscr{F}h)_{s(\beta)} \,\, \psi^{D,\alpha}_{\beta} $\\
&  (since $\bar{\mathscr{F}h} = (\mathscr{F}h, \mathscr{F}h')$ is a morphism in $Rep_{\!_{(\mathit{Q},\mathit{Q}')}})$\\
 & $=\zeta'^{\!_{D}}_{s(\beta)} \,\, \psi^{D,\alpha}_{\beta} $\\
&  (since $(C',(\zeta^{\!_{D}})_{\!_{{D} \in \mathscr{D}}})$ is a colimit of  $\mathcal{P}_2 \mathcal{U} \mathscr{F}$) \\
 \end{tabular}\\
 
\vspace{.2cm}

Therefore, $(C'_{s(\beta)}, (\zeta'^{\!_{D}}_{s(\beta)} \,\, \psi^{D,})_{\!_{{D} \in \mathscr{D}}}) $ is a cocone on  $\upsilon$. Since $Vec_k$ is cocomplete, there exists a unique $k$-linear map $\psi'^{\alpha}_{\beta}: C_{t(\alpha)} \rightarrow C'_{s(\beta)}$ with $  \psi'^{\alpha}_{\beta} \,\, \zeta^{\!_{D}}_{t(\alpha)} = \zeta'^{\!_{D}}_{s(\beta)} \,\, \psi^{D,\alpha}_{\beta}$ for every $D \in \mathscr{D}$  and for each pair of arrows $(\alpha,\beta) \in \mathrm{Q}_{\!_{1}} \times \in \mathit{Q}'_{\!_{1}}$. 

It turns out that  $\bar{\zeta}^{D} $ is a morphism in $Rep_{\!_{(\mathit{Q},\mathit{Q}')}}$ for any $D  \in \mathscr{D}$.
Let $\bar{C} =(C,C',\psi')$, $\bar{\zeta}^{D} = (\zeta^{D},\zeta'^{D})$. We claim that $(\bar{C},(\bar{\zeta}^{\!_{D}})_{\!_{{D} \in \mathscr{D}}})$ is a colimit of $\mathscr{F}$. To substantiate this claim, we need to show that for any cocone $ (\bar{\tilde{C}}, (\bar{\tilde{\zeta}}^{\!_{D}})_{\!_{{D} \in \mathscr{D}}})$, there exists a unique morphism $\bar{\Lambda} = (\Lambda,\Lambda')$ in $Rep_{\!_{(\mathit{Q},\mathit{Q}')}}$ with $ \bar{\Lambda} \,\, \bar{\zeta}^D  = \bar{\tilde{\zeta}}^D$ for every $D  \in \mathscr{D}$. Let $ (\bar{\tilde{C}}, (\bar{\tilde{\zeta}}^{\!_{D}})_{\!_{{D} \in \mathscr{D}}})$ be a cocone on  $\mathscr{F}$ and write $\bar{\tilde{C}} = (\tilde{C},\tilde{C}',\tilde{\psi}')$. Since $Rep_k(\mathit{Q}) ,\,\ Rep_k(\mathit{Q}')$ are cocomplete categories, exists unique morphisms $ \Lambda, \,\, \Lambda'$ in $Rep_k(\mathit{Q}) ,\,\ Rep_k(\mathit{Q}')$ respectively such that $ \Lambda\,\, \zeta^D = \tilde{\zeta}^D$, $ \Lambda' \,\, \zeta'^D  = \tilde{\zeta}'^D$ for every $D  \in \mathscr{D}$. It remains to show that $\bar{\Lambda} = (\Lambda,\Lambda')$ is a morphism in $Rep_{\!_{(\mathit{Q},\mathit{Q}')}}$. For any $D \in \mathscr{D}$, we have\\

\begin{tabular}{lllll}
$\tilde{\psi}'^{\alpha}_{\beta} \,\, \Lambda_{t(\alpha)} \,\, \zeta^{D}_{t(\alpha)} $ &  $= \tilde{\psi}'^{\alpha}_{\beta} \,\, \tilde{\zeta}^{D}_{t(\alpha)} $\\
&  (since $ \bar{\Lambda} \,\, \bar{\zeta}^D  = \bar{\tilde{\zeta}}^D$)\\
 & $= \tilde{\zeta}'^{D}_{s(\beta)} \,\, \psi^{D,\alpha}_{\beta}$\\
&  (since $\bar{\tilde{\zeta}}^{D} $ is a morphism in $Rep_{\!_{(\mathit{Q},\mathit{Q}')}}$)\\
& $= \Lambda'_{s(\beta)} \,\, \zeta'^{D}_{s(\beta)} \,\, \psi^{D,\alpha}_{\beta}$\\
& (since $ \bar{\Lambda} \,\, \bar{\zeta}^D  = \bar{\tilde{\zeta}}^D$)\\
& $= \Lambda'_{s(\beta)} \,\, \psi'^{\alpha}_{\beta}  \,\, \zeta^{D}_{t(\alpha)}$ \\
& (since $\bar{\zeta}^{D} $ is a morphism in $Rep_{\!_{(\mathit{Q},\mathit{Q}')}}$)\\
 \end{tabular}\\
 
\vspace{.2cm}

Notably,  $(C_{t(\alpha)}, (\zeta^{D}_{t(\alpha)})_{\!_{{D} \in \mathscr{D}}}) $ can be viewed as a colimit of  a functor $\upsilon': \mathscr{D} \rightarrow Vect_k$. Furthermore, $(\tilde{C}'_{s(\beta)}, (\tilde{\psi}'^{\alpha}_{\beta} \,\, \tilde{\zeta}^{D}_{t(\alpha)})_{\!_{{D} \in \mathscr{D}}}) $  is a cocone on $\upsilon'$ since for any $D, \,\, D'  \in \mathscr{D}$, we have\\

\begin{tabular}{lllll}
$ \tilde{\psi}'^{\alpha}_{\beta} \,\, \tilde{\zeta}^{D'}_{t(\alpha)} \,\, (\mathscr{F}h)_{t(\alpha)}$  &  $ = \tilde{\psi}'^{\alpha}_{\beta} \,\, \Lambda_{t(\alpha)} \,\, \zeta^{D'}_{t(\alpha)} \,\, (\mathscr{F}h)_{t(\alpha)}$ \\
&  (since $ \bar{\Lambda} \,\, \bar{\zeta}^D  = \bar{\tilde{\zeta}}^D$)\\
 &  $ = \tilde{\psi}'^{\alpha}_{\beta} \,\, \Lambda_{t(\alpha)} \,\, \zeta^{D}_{t(\alpha)} $ \\
&  (since $(C,(\zeta^{\!_{D}})_{\!_{{D} \in \mathscr{D}}})$ is a colimit of $\mathcal{P}_1 \mathcal{U} \mathscr{F}$) \\
 & $ = \tilde{\psi}'^{\alpha}_{\beta} \,\,  \tilde{\zeta}^{D}_{s(\beta)}$ \\
 &  (since $ \bar{\Lambda} \,\, \bar{\zeta}^D  = \bar{\tilde{\zeta}}^D$)\\
 \end{tabular}\\
 
\vspace{.2cm}

Since $Vec_k$ is cocomplete,  it follows from the universal property of the colimit that $\Lambda'_{s(\beta)} \,\, \tilde{\psi}'^{\alpha}_{\beta} = \psi'^{\alpha}_{\beta} \,\, \Lambda_{t(\alpha)}$. Consequently,  $\bar{\Lambda}$ is a morphism in $Rep_{\!_{(\mathit{Q},\mathit{Q}')}}$, which completes the proof.\\

\end{proof}

From Proposition (\ref{p.8}) and Remark (\ref{r.inductive}), we obtain.\\

\begin{proposition} \label{p.8.1}  
The category $Rep_{\!_{(\mathit{Q}_1,\mathit{Q}_2,...,\mathit{Q}_n)}}$ is cocomplete and the forgetful functor 
\begin{center}
•$\mathcal{U}_n: Rep_{\!_{(\mathit{Q}_1,\mathit{Q}_2,...,\mathit{Q}_n)}}  \rightarrow  Rep_k(\mathit{Q}_1) \times Rep_k(\mathit{Q}_2) \times . . . \times Rep_k(\mathit{Q}_n)$
\end{center}
 is cocontinuous. Furthermore, the colimit of objects in $Rep_{\!_{(\mathit{Q}_1,\mathit{Q}_2,...,\mathit{Q}_n)}}$ can be obtained by the corresponding construction for objects in $Rep_k(\mathit{Q}_1) \times Rep_k(\mathit{Q}_2) \times . . . \times Rep_k(\mathit{Q}_n)$.\\  
\end{proposition} 

\begin{proof}

\end{proof}

 \section{\textbf{$n$-quivers and $n$-representations} \label{s.$n$-quiivers}} 

Let $\bar{V} = (V,V', \psi)$, 
$\bar{W} = (W,W', \psi')$ be birepresentations of $(\mathit{Q},\mathit{Q}')$. For simplicity, we suppress $k$ from the tensor product $\otimes_k$ and use $\otimes$ instead. Define 
\begin{equation} \label{diag.eq29}
\bar{V} \otimes \bar{W} = (V \otimes W,V'\otimes W', \psi \otimes \psi') \,\,\,\,\,\,\,\,\,\ .
\end{equation}
Then $\bar{V} \otimes \bar{W}$ is clearly a birepresentation of $(\mathit{Q},\mathit{Q}')$. Let $\bar{f}: \bar{V} \rightarrow \bar{W}$, $\bar{g}: \bar{M} \rightarrow \bar{N}$ be morphisms in $Rep_{\!_{(\mathit{Q},\mathit{Q}')}}$ and write $\bar{V} = (V,V', \psi)$, $\bar{W} = (W,W', \psi')$,  $\bar{M} = (M,M', \Psi)$,  $\bar{N} = (N,N', \Psi')$ . Define  
\begin{equation} \label{diag.eq30}
\bar{f} \otimes \bar{g}: \bar{V} \otimes \bar{M} \rightarrow \bar{W} \otimes \bar{N} \,\,\,\,\,\,\,\,\,\ .
\end{equation}
Then it is clear that $\bar{f} \otimes \bar{g}$ is a morphism in $Rep_{\!_{(\mathit{Q},\mathit{Q}')}}$, and hence the following diagram is commutative. 

\begin{equation} \label{diag.eq31}
\xymatrix{ 
	& V_{s(\alpha)} \otimes  M_{s(\alpha)} \ar[dl]_{f_{s(\alpha)} \otimes g_{s(\alpha)}} \ar[r]^{\phi_{\alpha} \otimes \varphi_{\alpha}} & V_{t(\alpha)} \otimes M_{t(\alpha)} \ar[dl]^{f_{t(\alpha)} \otimes g_{t(\alpha)}} \ar[dd]^{\psi^{\alpha}_{\beta} \otimes \Psi^{\alpha}_{\beta}}&& \\
	W_{s(\alpha)} \otimes  N_{s(\alpha)} \ar[r]_{\phi'_{\alpha} \otimes \varphi'_{\alpha}} & W_{t(\alpha)} \otimes N_{t(\alpha)} \ar[dd]_{\psi'^{\alpha}_{\beta} \otimes \Psi'^{\alpha}_{\beta}}&& &\\
	 && V'_{s(\alpha)} \otimes  M'_{s(\alpha)} \ar[dl]_{f'_{s(\beta)} \otimes g'_{s(\beta)}} \ar[r]^{\mu_{\beta} \otimes \nu_{\beta}}& V'_{t(\alpha)} \otimes M'_{t(\alpha)} \ar[dl]^{f'_{t(\beta)} \otimes g'_{t(\beta)}} \\
	 &W'_{s(\alpha)} \otimes  N'_{s(\alpha)} \ar[r]_{\mu'_{\beta} \otimes \nu'_{\beta}}& W'_{t(\alpha)} \otimes N'_{t(\alpha)} &
}
\end{equation}

\vspace{.2cm}

Thus, the category $Rep_{\!_{(\mathit{Q},\mathit{Q}')}}$ is a monoidal category, and hence by Remark (\ref{r.inductive}), $Rep_{\!_{(\mathit{Q}_1,\mathit{Q}_2,.\,.\,.,\,\mathit{Q}_n)}}$ is a monoidal category for any  $n \geq 2$. For the basic notions of monoidal categories, we refer the reader to \cite{Etingof}, \cite{Street}, \cite{Bakalov}, and \cite[Chapter 6]{Borceux2}.\\

\begin{definition} \label{def.5} 
Let $\mathit{Q} = (\mathrm{Q}_{\!_{0}}, \mathrm{Q}_{\!_{1}}, s,t)$, $\mathit{Q}' = (\mathrm{Q}'_{\!_{0}}, \mathrm{Q}'_{\!_{1}}, s',t')$ be quivers, and let  $\{\varrho^{\alpha}_{\beta} \,\,\,\,\,: \,\, \alpha \in \mathrm{Q}_{\!_{1}},\,\, \beta \in \mathrm{Q}'_{\!_{1}} \} $ be a collection of of arrows 
$t(\alpha) \xrightarrow{\varrho^{\alpha}_{\beta}} s'(\beta)$ one for each pair of arrows $(\alpha,\beta) \in \mathrm{Q}_{\!_{1}} \times \mathrm{Q}'_{\!_{1}}$. \\

A \textbf{$2$-quiiver induced by $(\mathit{Q},\mathit{Q}')$} is a quiver $\mathscr{Q}_{\!_{(\mathit{Q},\mathit{Q}')}} = (\tilde{Q}_{\!_{0}}, \tilde{Q}_{\!_{1}}, s'',t'')$, where 

 \begin{itemize}
  \item $\tilde{\mathrm{Q}}_{\!_{0}} = \mathrm{Q}_{\!_{0}} \,\, \sqcup \,\, \mathrm{Q}'_{\!_{0}}$, \\
   \item $\tilde{\mathrm{Q}}_{\!_{1}} = \mathrm{Q}_{\!_{1}} \,\, \sqcup \,\, \mathrm{Q}'_{\!_{1}} \,\, \sqcup \,\, \{\varrho^{\alpha}_{\beta} \,\,\,\,\,: \,\, \alpha \in \mathrm{Q}_{\!_{1}},\,\, \beta \in \mathrm{Q}'_{\!_{1}} \} $,\\
  \item  $s'':\tilde{\mathrm{Q}}_{\!_{1}} \rightarrow \tilde{\mathrm{Q}}_{\!_{0}}$ a map from arrows to vertices, mapping an arrow to its starting
point, \\
\item $t'':\tilde{\mathrm{Q}}_{\!_{1}} \rightarrow \tilde{\mathrm{Q}}_{\!_{0}}$ a map from arrows to vertices, mapping an arrow to its terminal point.\\
\end{itemize}

The notation $\sqcup$ above denotes the disjoint union.  \\

The above definition turns out that one can inductively define $n$-quivers for any integer $n \geq 2$. \\
For any  $m \in \{2,\,\,..., \,\,n\}$, let $\mathit{Q}_m = (\mathrm{Q}^{(m)}_{\!_{0}}, \mathrm{Q}^{(m)}_{\!_{1}}, s^{(m)},t^{(m)})$ be a quiver, and let $\{\varrho_{\!_{m}\gamma^{(m-1)}}^{\gamma^{(m)}} \,\,\,\,\,: \,\,\,\,\,\, \gamma^{(m)} \in \mathrm{Q}^{(m)}_{\!_{1}}, \,\,\,\, m \in \{2,\,\,..., \,\,n\} \}$ be a collection of of arrows \\
\begin{center}
•$t^{(m-1)} \xrightarrow{\varrho_{\!_{m}\gamma^{(m-1)}}^{\gamma^{(m)}}} s^{(m)}$ 
\end{center}

\vspace{.25cm}

one for each pair of arrows  $(\gamma^{(m-1)},\gamma^{(m)}) \in \mathrm{Q}^{(m-1)}_{\!_{1}} \times  \mathrm{Q}^{(m)}_{\!_{1}}$. 
 An \textbf{$n$-quiver induced by} $(\mathit{Q}_1, \,\, \mathit{Q}_2, \,\,... \,\, ,  \mathit{Q}_n)$ is
a quiver $\mathscr{Q}_{\!_{(\mathit{Q}_1, \mathit{Q}_2,...,  \mathit{Q}_n)}} = (\hat{\mathrm{Q}_{\!_{0}}}, \hat{\mathrm{Q}_{\!_{1}}}, \hat{s},\hat{t})$, where \\
 \begin{itemize}
  \item $\hat{\mathrm{Q}_{\!_{0}}} = \sqcup_{m=1}^{n} \mathrm{Q}^{(m)}_{\!_{0}}$, \\
   \item $\hat{\mathrm{Q}_{\!_{1}}} = \sqcup_{m=1}^{n} \mathrm{Q}^{(m)}_{\!_{1}} \,\, \sqcup \,\, \{\varrho_{\!_{m}\gamma^{(m-1)}}^{\gamma^{(m)}} \,\,\,\,\,: \,\,\,\,\,\, \gamma^{(m)} \in \mathrm{Q}^{(m)}_{\!_{1}}, \,\,\,\, m \in \{2,\,\,..., \,\,n\} \} $,\\
  \item  $\hat{s}: \hat{\mathrm{Q}_{\!_{1}}} \rightarrow \hat{\mathrm{Q}_{\!_{0}}}$ a map from arrows to vertices, mapping an arrow to its starting
point, \\
\item $\hat{t}: \hat{\mathrm{Q}_{\!_{1}}} \rightarrow \hat{\mathrm{Q}_{\!_{0}}}$ a map from arrows to vertices, mapping an arrow to its terminal point.\\
\end{itemize}

\end{definition} 

\begin{example}

\begin{equation} \label{diag.eq32}
\xymatrix{
&&\\
\mathit{Q}: & 1 \ar[rr]^{} && 2 
}
\hspace{30pt}
\xymatrix{
&&\\
\mathit{Q}': & 3  \ar@/_2.6pc/[rr]^{} \ar[rr]^{} \ar@/^2.6pc/[rr]^{}  & & 4 
}
\hspace{30pt}
\xymatrix{
&5 \ar[dr]^{} & &\\
\mathit{Q}'':&&6  &7 \ar[l]^{}\\
&8  \ar[ur]^{} &&
}
\end{equation}

\begin{equation} \label{diag.eq33}
\xymatrix{
&&\\
\mathscr{Q}_{\!_{(\mathit{Q},\mathit{Q}')}} : & 1 \ar[rr]^{} & &2 
\ar@/_2.6pc/[rrrr]^{}  \ar@/^2.6pc/[rrrr]^{} \ar[rrrr]^{}
 & &&& 3  \ar@/_2.6pc/[rr]^{} \ar[rr]^{} \ar@/^2.6pc/[rr]^{}  & & 4
}
\end{equation}

\begin{equation} \label{diag.eq33.1}
\xymatrix{
&&&&&&&&&5  \ar@/^1.pc/[ddr]^(.4){} & &\\
       &&&&&&&&&\\
\mathscr{Q}_{\!_{(\mathit{Q}'',\mathit{Q},\mathit{Q}')}}: &4  & & 3  \ar@/_2.6pc/[ll]_{} \ar[ll]_{} \ar@/^2.6pc/[ll]^{} &&& 2  \ar@/_5pc/[lll]_{} \ar[lll]_{} \ar@/^5pc/[lll]^{}   
        &1 \ar[l]_{}  &&&6 \ar@/_3pc/[lll]_(.6){} 
              \ar@/^3pc/[lll]^(.6){}
               \ar[lll]_(.5){}
              &7 \ar[l]_{} & &\\
         &&&&&&&&&\\
&&&&&&&&&8  \ar@/_1.pc/[uur]_{} &&&&&
}
\end{equation}

We have $k(\mathit{Q}) \cong \begin{bmatrix}
       k & k           \\[0.3em]
       0 & k        \\[0.3em]
             \end{bmatrix} $,  $k(\mathit{Q}') \cong \begin{bmatrix}
       k & k^3          \\[0.3em]
       0 & k            \\[0.3em]
                   \end{bmatrix} $,  $k(\mathscr{Q}_{\!_{(\mathit{Q},\mathit{Q}')}}) \cong \begin{bmatrix}
       k &  k & k^3 & k^9         \\[0.3em]
       0 &  k & k^3 & k^9         \\[0.3em]
       0 & 0 & k & k^3         \\[0.3em]
       0 &  0 & 0 & k         \\[0.3em]
                   \end{bmatrix} $, and  
                      
$k(\mathscr{Q}_{\!_{(\mathit{Q}'',\mathit{Q},\mathit{Q}')}})
\cong \begin{bmatrix}
       k &  k & k^3 & k^9 & 0 & 0 & 0 & 0         \\[0.3em]
       0 &  k & k^3 & k^9 & 0 & 0 & 0 & 0        \\[0.3em]
       0 & 0 & k & k^3    & 0 & 0 & 0 & 0       \\[0.3em]
       0 &  0 & 0 & k     & 0 & 0 & 0 & 0    \\[0.3em]
       k^3 &  k^3 & k^9 & k^{27}     & k & k & 0 & 0    \\[0.3em]
       k^3 &  k^3 & k^9 & k^{27}     & 0 & k & 0 & 0    \\[0.3em]
       k^3 &  k^3 & k^9 & k^{27}     & 0 & k & k & 0    \\[0.3em]
       k^3 &  k^3 & k^9 & k^{27}     & 0 & k & 0 & k    \\[0.3em]
                   \end{bmatrix}  $. \\
                   
 \vspace{.2cm}                   
                   
Consider the following:

\begin{equation} \label{diag.eq34}
\xymatrix{
&&\\
V: \,\, \,\, k \ar[rr]^{1} && k 
}
\hspace{30pt}
\xymatrix{
&&\\
V': \,\, \,\, k  \ar@/_2.6pc/[rr]^{1} \ar[rr]^{1} \ar@/^2.6pc/[rr]^{1}  & & k 
}
\hspace{30pt}
\xymatrix{
&k \ar[dr]^{\begin{bmatrix}
       1            \\[0.3em]
       0         \\[0.3em]
       \end{bmatrix}} & &\\
V'': &&k^2  &k \ar[l]_{\begin{bmatrix}
       1            \\[0.3em]
       1         \\[0.3em]
       \end{bmatrix}}\\
&k  \ar[ur]_{\begin{bmatrix}
       0            \\[0.3em]
       1         \\[0.3em]
       \end{bmatrix}} &&
}
\end{equation}

Then $V, \, V', \, V''$  are clearly representations of $\mathit{Q}$, $\mathit{Q}'$, $\mathit{Q}''$ respectively. \\
Now, consider the following:
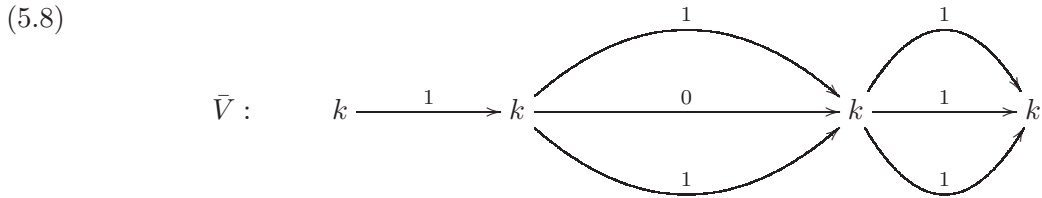
\begin{equation} \label{diag.eq35}
\xymatrix{
&&\\
\bar{V} : & k \ar[rr]^{1} & &k 
\ar@/_2.6pc/[rrrr]^{1}  \ar@/^2.6pc/[rrrr]^{1} \ar[rrrr]^{0}
 & &&& k  \ar@/_2.6pc/[rr]^{1} \ar[rr]^{1} \ar@/^2.6pc/[rr]^{1}  & & k
}
\end{equation}

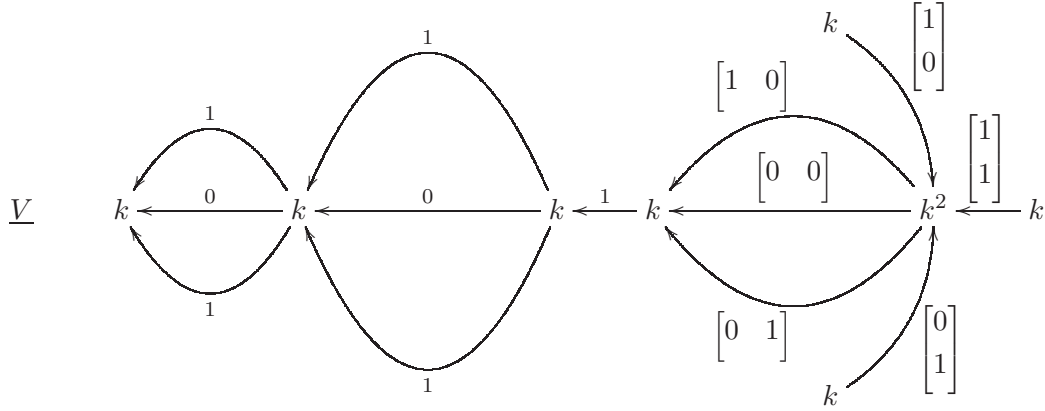
\begin{equation} \label{diag.eq35.1}
\xymatrix{
&&&&&&&&&k  \ar@/^1.pc/[ddr]^(.4){\begin{bmatrix}
       1            \\[0.3em]
       0         \\[0.3em]
       \end{bmatrix}} & &\\
       &&&&&&&&&\\
\underline{V} &k  & & k  \ar@/_2.6pc/[ll]_{1} \ar[ll]_{0} \ar@/^2.6pc/[ll]^{1} &&& k  \ar@/_5pc/[lll]_{1} \ar[lll]_{0} \ar@/^5pc/[lll]^{1}   
        &k \ar[l]_{1}  &&&k^2 \ar@/_3pc/[lll]_(.6){\begin{bmatrix}
       1 & 0       \\[0.3em]
              \end{bmatrix}} 
              \ar@/^3pc/[lll]^(.6){\begin{bmatrix}
       0 & 1       \\[0.3em]
              \end{bmatrix}}
               \ar[lll]_(.5){\begin{bmatrix}
       0 & 0       \\[0.3em]
              \end{bmatrix}}
              &k \ar[l]_{\begin{bmatrix}
       1            \\[0.3em]
       1         \\[0.3em]
       \end{bmatrix}} & &\\
         &&&&&&&&&\\
&&&&&&&&&k  \ar@/_1.pc/[uur]_{\begin{bmatrix}
       0            \\[0.3em]
       1         \\[0.3em]
       \end{bmatrix}} &&&&&
}
\end{equation}\\

Then $\bar{V}$ is a representations of $\mathscr{Q}_{\!_{(\mathit{Q},\mathit{Q}')}}$ and $\underline{V}$ is a representation of $\mathscr{Q}_{\!_{(\mathit{Q}'',\mathit{Q},\mathit{Q}')}}$.\\
\end{example}

\begin{remark} \textbf{•} \label{r.n.quivers.interms.components} 

\begin{enumerate}[label=(\roman*)]
\item In general, if $\Omega_n = \{\varrho_{\!_{n}\gamma^{(n-1)}}^{\gamma^{(n)}} \,\,\,\,\,: \,\,\,\,\,\, \gamma^{(n-1)}\in \mathrm{Q}^{(n-1)}_{\!_{1}}, \,\,\, \gamma^{(n)} \in \mathrm{Q}^{(n)}_{\!_{1}}\} $, then we have\\

  $$ k\mathscr{Q}_{\!_{(\mathit{Q}_1, \mathit{Q}_2,...,  \mathit{Q}_n)}} \cong 
\begin{bmatrix}

  k\mathscr{Q}_{\!_{(\mathit{Q}_1, \mathit{Q}_2,...,  \mathit{Q}_{n-1})}}  &  &  k \Omega_n \\
  &  &\\
   0 &  & k\mathit{Q}_n\\
\end{bmatrix} \,\,\,\,\,\,\,\,\,\,\, \text{or} \,\,\,\,\,\,\,\,\,\,\,\, k\mathscr{Q}_{\!_{(\mathit{Q}_1, \mathit{Q}_2,...,  \mathit{Q}_n)}} \cong 
\begin{bmatrix}

  k\mathscr{Q}_{\!_{(\mathit{Q}_1, \mathit{Q}_2,...,  \mathit{Q}_{n-1})}}  &  & 0 \\
  &  &\\
    k \Omega_n &  & k\mathit{Q}_n\\
\end{bmatrix}
$$  \\

where $k \Omega $ is the subspace of $ k\mathscr{Q}_{\!_{(\mathit{Q}_1, \mathit{Q}_2,...,  \mathit{Q}_n)}} $ generated by the set $\Omega_n$. To avoid confusion, we identify $k\mathscr{Q}_{\!_{(\mathit{Q}_1)}}$ as  $k\mathcal{Q}_1$.\\

\item We note that $\bar{V}$ in the previous example can be identified as a birepresentations of $(\mathit{Q},\mathit{Q}')$. Indeed, every representation of $\mathscr{Q}_{\!_{(\mathit{Q},\mathit{Q}')}}$ can be identified as a birepresentation of $(\mathit{Q},\mathit{Q}')$. Conversely, every birepresentation of $(\mathit{Q},\mathit{Q}')$ can be viewed as a representation of $\mathscr{Q}_{\!_{(\mathit{Q},\mathit{Q}')}}$. Similarly, $\underline{V}$ can be identified as a $3$-representation of $\mathscr{Q}_{\!_{(\mathit{Q}'',\mathit{Q},\mathit{Q}')}}$.  This clearly gives rise to the following proposition. \\
\end{enumerate}
\end{remark}

\begin{proposition} \label{p.9}  
There exists an equivalence of categories $Rep_{\!_{(\mathit{Q},\mathit{Q}')}}  \simeq Rep_k(\mathscr{Q}_{\!_{(\mathit{Q},\mathit{Q}')}})$ that restricts to an equivalence  $rep_{\!_{(\mathit{Q},\mathit{Q}')}}  \simeq rep_k(\mathscr{Q}_{\!_{(\mathit{Q},\mathit{Q}')}})$.\\

\end{proposition} 

\begin{proof}
This follows from the construction of $\mathscr{Q}_{\!_{(\mathit{Q},\mathit{Q}')}}$ in Definition ((\ref{def.5})). In fact, there is a decomposing functor $ \mathcal{F}: Rep_{\!_{(\mathit{Q},\mathit{Q}')}} \rightarrow  Rep_k(\mathscr{Q}_{\!_{(\mathit{Q},\mathit{Q}')}})$  which identify any representation of $\mathscr{Q}_{\!_{(\mathit{Q},\mathit{Q}')}}$ as  $\mathit{Q},\,\,\, \mathit{Q}'$ and a collection of compatibility maps and assign it to a birepresentation of $(\mathit{Q},\mathit{Q}')$. Its inverse is the gluing functor $\mathcal{G}: Rep_k(\mathscr{Q}_{\!_{(\mathit{Q},\mathit{Q}')}})  \rightarrow Rep_{\!_{(\mathit{Q},\mathit{Q}')}} $, which view each birepresentation as three parts (two representations of $\mathit{Q},\,\,\, \mathit{Q}'$ and a collection of compatibility maps) and assemble them for a representation of $\mathscr{Q}_{\!_{(\mathit{Q},\mathit{Q}')}}$.\\

\end{proof}

Using Induction and Remark (\ref{r.inductive}), we have the following.\\

\begin{proposition} \label{p.9.1}  
For any $n \geq 2$, there exists an equivalence of categories $Rep_{\!_{(\mathit{Q}_1,\mathit{Q}_2,...,\mathit{Q}_n)}} \simeq Rep_k(\mathscr{Q}_{\!_{(\mathit{Q}_1, \,\, \mathit{Q}_2,...,  \mathit{Q}_n)}})$ that restricts to an equivalence  $rep_{\!_{(\mathit{Q}_1,\mathit{Q}_2,...,\,\mathit{Q}_n)}}  \simeq rep_k(\mathscr{Q}_{\!_{(\mathit{Q}_1,\mathit{Q}_2,...,\mathit{Q}_n)}})$.\\

\end{proposition} 

\begin{proof}

\end{proof}

The following is an immediate consequence of Proposition (\ref{p.9}) and Proposition (\ref{p.Equiv.Rep.Mod}).\\

\begin{proposition} \label{p.10}  
There exists an equivalence of categories $Mod \, k \mathscr{Q}_{\!_{(\mathit{Q},\mathit{Q}')}} \simeq Rep_{\!_{(\mathit{Q},\mathit{Q}')}}$ that restricts to an equivalence $mod \, k \mathscr{Q}_{\!_{(\mathit{Q},\mathit{Q}')}} \simeq rep_{\!_{(\mathit{Q},\mathit{Q}')}}$.  \\
\end{proposition} 

\begin{proof}

\end{proof}

By Remark (\ref{r.inductive}) and Propositions (\ref{p.9.1}),  (\ref{p.Equiv.Rep.Mod}), we have the following.\\

\begin{proposition} \label{p.10.1}  
For any $n \geq 2$, there exists an equivalence of categories $Mod \, k \mathscr{Q}_{\!_{(\mathit{Q}_1,\mathit{Q}_2,...,\mathit{Q}_n)}} \simeq Rep_{\!_{(\mathit{Q}_1,\mathit{Q}_2,...,\mathit{Q}_n)}}$ that restricts to an equivalence $mod \, k \mathscr{Q}_{\!_{(\mathit{Q}_1,\mathit{Q}_2,...,\mathit{Q}_n)}} \simeq rep_{\!_{(\mathit{Q}_1,\mathit{Q}_2,...,\mathit{Q}_n)}}$.  \\
\end{proposition} 

\begin{proof}

\end{proof}

We recall the definitions of a co-wellpowered category and a generating set for a category. \\
Let $\mathfrak{E}$ be a class of all epimorphisms of a category $\mathfrak{A}$. Then $\mathfrak{A}$ is called \textit{co-wellpowered} provided that no $\mathfrak{A}$-object has a proper class of pairwise non-isomorphic quotients \cite[p. 125]{Adamek}.  In other words, for every object the quotients form a set \cite[p. 92, 95]{Schubert}. We refer the reader to \cite{Adamek} basics on quotients and co-wellpowered categories.\\

Following  \cite[p. 127]{Mac Lane1}, a set $\mathcal{G}$  of objects of the category $\mathscr{C}$  is said to \textit{generate} $\mathscr{C}$  when any parallel pair $f,g: X \rightarrow Y$  of arrows of $\mathscr{C}$, $ f \neq  g $  implies that there is an $G \in \mathcal{G}$  and an arrow $\alpha:G \rightarrow X$ in $\mathscr{C}$ with $f\alpha \neq g\alpha$  (the term ``generates" is well established but poorly chosen; ``\textit{separates}" would have been better).  For the basic concepts of generating sets, we refer to   \cite{Mac Lane1},  \cite{Adamek}, or \cite{Freyd}. \\

The following proposition immediately follows from  Remark (\ref{r.inductive}), Proposition (\ref{p.10.1}) and the fact that the categories of modules are co-wellpowered with generating sets. 

\begin{proposition} \label{p.11}  
For any $n \geq 2$, the category $Rep_{\!_{(\mathit{Q}_1,\mathit{Q}_2,...,\mathit{Q}_n)}}$  is  co-wellpowered.  \\
\end{proposition} 

\begin{proof}

\end{proof} 

\begin{proposition} \label{p.12}  
For any $n \geq 2$, the category $Rep_{\!_{(\mathit{Q}_1,\mathit{Q}_2,...,\mathit{Q}_n)}}$ has a generating set.  \\
\end{proposition} 

\begin{proof}

\end{proof}

Using Theorem (\ref{p.SAFT}), we have the following.

\begin{proposition} \label{p.13}  
For any $n \geq 2$, the obvious forgetful functor $\mathcal{U}_n: Rep_{\!_{(\mathit{Q}_1,\mathit{Q}_2,...,\mathit{Q}_n)}}  \rightarrow  Rep_k(\mathit{Q}_1) \times Rep_k(\mathit{Q}_2) \times . . . \times Rep_k(\mathit{Q}_n)$ has a right adjoint. Equivalently, for any $n \geq 2$, the concrete category $(Rep_{\!_{(\mathit{Q}_1,\mathit{Q}_2,...,\mathit{Q}_n)}} , \mathcal{U}_n)$ has cofree objects. \\
\end{proposition} 

\begin{proof}

\end{proof}

Let  $CoAlg(Rep_{\!_{(\mathit{Q},\mathit{Q}')}})$ be the category of coalgebras in $Rep_{\!_{(\mathit{Q},\mathit{Q}')}}$. By \cite[p. 30]{Abdulwahid}, the left $k \mathscr{Q}_{\!_{(\mathit{Q},\mathit{Q}')}}$-module coalgebras which are finitely generated as left $k \mathscr{Q}_{\!_{(\mathit{Q},\mathit{Q}')}}$-modules form a system of generators for $CoAlg(Mod \, k \mathscr{Q}_{\!_{(\mathit{Q},\mathit{Q}')}})$. Thus, from Proposition (\ref{p.cowp}) and Theorem (\ref{p.SAFT}), we have the following.\\

\begin{proposition} \label{p.14}  
The forgetful functor ${\mathscr{U}}: CoAlg(Rep_{\!_{(\mathit{Q},\mathit{Q}')}})  \rightarrow  Rep_{\!_{(\mathit{Q},\mathit{Q}')}} $ has a right adjoint. Equivalently, the concrete category $(Rep_{\!_{(\mathit{Q},\mathit{Q}')}} , \mathscr{U})$ has cofree objects. \\
\end{proposition} 

\begin{proof}

\end{proof}

Remark (\ref{r.inductive}) implies the following consequence.\\

\begin{proposition} \label{p.14.1}  
For any $n \geq 2$, the forgetful functor $\mathscr{U}_n: CoAlg(Rep_{\!_{(\mathit{Q}_1,\mathit{Q}_2,...,\mathit{Q}_n)}})  \rightarrow  Rep_{\!_{(\mathit{Q}_1,\mathit{Q}_2,...,\mathit{Q}_n)}} $ has a right adjoint. Equivalently, for any $n \geq 2$, the concrete category $(Rep_{\!_{(\mathit{Q}_1,\mathit{Q}_2,...,\mathit{Q}_n)}} , \mathscr{U}_n)$ has cofree objects.\\

\end{proposition} 

\begin{proof}

\end{proof}

\begin{remark} \label{r.cofree.birep}
If the forgetful functor $\mathscr{U}: CoAlg(Rep_{\!_{(\mathit{Q},\mathit{Q}')}})\rightarrow Rep_{\!_{(\mathit{Q},\mathit{Q}')}}$  has a right adjoint $\mathscr{V}$ and $\bar{M} \in Rep_{\!_{(\mathit{Q},\mathit{Q}')}}$, then the cofree coalgebra over $\bar{M}$ can be given by 
$$\mathscr{V}(\bar{M})=\lim\limits_{\stackrel{\longrightarrow}{[\bar{f}:\bar{G}\rightarrow \bar{M}]\in Rep_{\!_{(\mathit{Q},\mathit{Q}')}},\, \bar{G} \in CoAlg(Rep_{\!_{(\mathit{Q},\mathit{Q}')}})},\, \bar{G} \in rep_{\!_{(\mathit{Q},\mathit{Q}')}}}\bar{G}.$$  Thus, the concrete category $( CoAlg(Rep_{\!_{(\mathit{Q},\mathit{Q}')}}), \mathscr{U})$ has cofree objects given in terms of colimits and generators. \\

It follows that, for any $n \geq 2$, if the forgetful functor $\mathscr{U}_n: CoAlg(Rep_{\!_{(\mathit{Q}_1,\mathit{Q}_2,...,\mathit{Q}_n)}})\rightarrow Rep_{\!_{(\mathit{Q}_1,\mathit{Q}_2,...,\mathit{Q}_n)}}$  has a right adjoint $\mathscr{V}_n$ and $\bar{M} \in Rep_{\!_{(\mathit{Q}_1,\mathit{Q}_2,...,\mathit{Q}_n)}}$, then the cofree coalgebra over $\bar{M}$ can be given by 
$$\mathscr{V}_n(\bar{M})=\lim\limits_{\stackrel{\longrightarrow}{[\bar{f}:\bar{G}\rightarrow \bar{M}]\in Rep_{\!_{(\mathit{Q}_1,\mathit{Q}_2,...,\mathit{Q}_n)}},\, \bar{G} \in CoAlg(Rep_{\!_{(\mathit{Q}_1,\mathit{Q}_2,...,\mathit{Q}_n)}})},\, \bar{G} \in rep_{\!_{(\mathit{Q}_1,\mathit{Q}_2,...,\mathit{Q}_n)}}}\bar{G}.$$ \\
\end{remark}

The following proposition explicitly describes colimits and cofree objects in the product categories.\\

\begin{proposition} \label{p.15}  
For any $m \in \{1, \,\, 2, \,\, 3, \,\, . . ., \,\, n\}$, let $\mathcal{U}_m: \mathscr{A}_m  \rightarrow  \mathscr{X}_m$ be a forgetful functor  with a right adjoint $\mathcal{V}_m$. Then the product functor  $\mathcal{V}_1 \times \mathcal{V}_2 \times . . . \times \mathcal{V}_n  : \mathscr{X}_1 \times \mathscr{X}_2 \times . . . \times \mathscr{X}_n  \rightarrow  \mathscr{A}_1 \times \mathscr{A}_2 \times . . . \times \mathscr{A}_n $ is a right adjoint of the obvious forgetful functor $\mathcal{U}_1 \times \mathcal{U}_2 \times . . . \times \mathcal{U}_n:   \mathscr{A}_1 \times \mathscr{A}_2 \times . . . \times \mathscr{A}_n  \rightarrow \mathscr{X}_1 \times \mathscr{X}_2 \times . . . \times \mathscr{X}_n $. Furthermore, for any $(X_1,X_2, . . ., X_n) \in \mathscr{X}_1 \times \mathscr{X}_2 \times . . . \times \mathcal{X}_n $, the cofree object over $(X_1,X_2, . . ., X_n)$ is exactly $(\mathcal{V}_1(X_1) \times \mathcal{V}_2(X_2) \times . . . \times \mathcal{V}_n(X_n))$. \\
\end{proposition} 

\begin{proof} 
 The proof is straightforward.\\
 
\end{proof}

For any $m \in \{1, \,\, 2, \,\, . . ., \,\, n\}$, let $\mathcal{U}_m: CoAlg(Rep_k(\mathit{Q}_m))  \rightarrow  Rep_k(\mathit{Q}_m)$ be the obvious  forgetful functor with a right adjoint $\mathcal{V}_m$. Proposition (\ref{p.15}) and Remark (\ref{r.inductive}) suggest the following immediate consequence.\\

\begin{corollary} \label{p.16}  
The product functor  $\mathcal{V}_1 \times \mathcal{V}_2 \times . . . \times \mathcal{V}_n  : Rep_k(\mathit{Q}_1) \times Rep_k(\mathit{Q}_2)\times . . . \times Rep_k(\mathit{Q}_n)  \rightarrow   CoAlg(Rep_k(\mathit{Q}_1)) \times CoAlg(Rep_k(\mathit{Q}_2))\times . . . \times CoAlg(Rep_k(\mathit{Q}_n)) $ is a right adjoint of the obvious product forgetful functor $\mathcal{U}_1 \times \mathcal{U}_2 \times . . . \times \mathcal{U}_n:   CoAlg(Rep_k(\mathit{Q}_1)) \times CoAlg(Rep_k(\mathit{Q}_2))\times . . . \times CoAlg(Rep_k(\mathit{Q}_n))   \rightarrow Rep_k(\mathit{Q}_1) \times Rep_k(\mathit{Q}_2)\times . . . \times Rep_k(\mathit{Q}_n) $. \\ 
Furthermore, for any $(M_1,M_2,...,M_n) \in  Rep_k(\mathit{Q}_1) \times Rep_k(\mathit{Q}_2)\times . . . \times Rep_k(\mathit{Q}_n)$, the cofree object over $(M_1,M_2,...,M_n)$ is exactly $(\mathcal{V}_1(M_1)$, $\mathcal{V}_2(M_2)$,...,$\mathcal{V}_n(M_n))$. \\
\end{corollary} 

\begin{proof}
 
\end{proof}

For any $m \in \{1, \,\,  2, \,\, . . ., \,\, n\}$, let $\mathcal{U}_m: CoAlg(Rep_k(\mathit{Q}_m))  \rightarrow  Rep_k(\mathit{Q}_m)$ be the obvious  forgetful functor with a right adjoint $\mathcal{V}_m$. By Proposition (\ref{p.12}), for any $n \geq 2$, the category $Rep_{\!_{(\mathit{Q}_1,\mathit{Q}_2,...,\mathit{Q}_n)}}$ has a generating set $\mathscr{G}$. For any $m \in \{1, \,\, 2, \,\, 3, \,\, . . ., \,\, n\}$, let $\mathcal{G}_m$ be a generating set of $Rep_k(\mathit{Q}_m)$. The definition of $n$-representations of quivers implies that each element of $\mathscr{G}$ takes the form $(G_1,G_2,...,G_n,\chi_1,\chi_2, ..., \chi_{n-1})$ where $G_m \in \mathcal{G}_m$ for every  $m \in \{1, \,\, 2, \,\, 3, \,\, . . ., \,\, n\}$.  Furthermore, Remark (\ref{r.cofree.birep}) and Proposition  (\ref{p.8}) imply the following proposition which gives an explicit description for cofree coalgebras in the concrete category $(Rep_{\!_{(\mathit{Q}_1,\mathit{Q}_2,...,\mathit{Q}_n)}} , \mathscr{U})$, where 
 ${\mathscr{U}}: CoAlg(Rep_{\!_{(\mathit{Q}_1,\mathit{Q}_2,...,\mathit{Q}_n)}})  \rightarrow  Rep_{\!_{(\mathit{Q}_1,\mathit{Q}_2,...,\mathit{Q}_n)}} $ is the obvious forgetful functor with a right adjoint $\mathscr{V}$. \\

\begin{proposition} \label{p.17}  
Let $\bar{M} = (M^{(1)},M^{(2)},...,M^{(n)},\psi_{\!_{1}}, \psi_{\!_{2}}, ..., \psi_{\!_{n-1}}) \in Rep_{\!_{(\mathit{Q}_1,\mathit{Q}_2,...,\mathit{Q}_n)}}$. Then the cofree object over $\bar{M}$ is exactly $(\mathcal{V}_1(M^{(1)}),\mathcal{V}_2(M^{(2)}), . . ., \mathcal{V}_n(M^{(n)}),{\psi}'_{\!_{1}}, {\psi}'_{\!_{2}}, ..., {\psi}'_{\!_{n-1}})$,  for some unique $k$-linear maps ${\psi}'_{\!_{1}}, {\psi}'_{\!_{2}}, ..., {\psi}'_{\!_{n-1}}$. \\
\end{proposition} 

\begin{proof}
This can be proved by applying Remark (\ref{r.cofree.birep}) and Proposition (\ref{p.8}). Explicitly, if  $\bar{G} =  (G^{(1)},G^{(2)},...,G^{(n)},\chi_1,\chi_2, ..., \chi_{n-1})$, , $\bar{f} = (f^{(1)},f^{(2)}),...,f^{(n)})$, we have the following.

\begin{tabular}{lllll}
 $\mathscr{V}(\bar{M})$ &  $=\lim\limits_{\stackrel{\longrightarrow}{[\bar{f}:\bar{G} \rightarrow \bar{M}]\in Rep_{\!_{(\mathit{Q}_1,\mathit{Q}_2,...,\mathit{Q}_n)}},\,\bar{G} \in CoAlg(Rep_{\!_{(\mathit{Q}_1,\mathit{Q}_2,...,\mathit{Q}_n)}})},\,\bar{G} \in rep_{\!_{(\mathit{Q}_1,\mathit{Q}_2,...,\mathit{Q}_n)}}} \bar{G}, \,\,\, $\\
 & \\ 
   & $= (\lim\limits_{\stackrel{\longrightarrow}{[f^{(1)}:G^{(1)} \rightarrow M^{(1)}]\in Rep_k(\mathit{Q}_1),\,G^{(1)} \in CoAlg(Rep_k(\mathit{Q}_1))},\,G^{(1)} \in rep_k(\mathit{Q}_1)} G^{(1)}, $ \\
  &   \\
  & $\,\,\,\,\,\,\,\,\,\,\,\, \,\,\,\,\,\,\,\,\,\,\,\, \,\,\,\,\,\,\,\,\,\,\,\,\,\,\,\,\,\,\,\,\,\,\,\,. \,\,\,\,\,\,\,\,\,\,\,\,\,\,\,\,\,\,\,\,\,\,\,\,\,\,\,\,\,\,\,\,\,\,\,\, . \,\,\,\,\,\,\,\,\,\,\,\,\,\,\,\,\,\,\,\,\,\,\,\,\,\,\,\,\,\,\,\,\,\,\,\,. \,\,\,\,\,\,\,\,\,\,\,\,\,\,\,\,\,\,\,\,\,\,\,\, \,\,\,\,\,\,\,\,\,\,\,\,,$ \\
  & \\
  &  $ \lim\limits_{\stackrel{\longrightarrow}{[f^{(n)}:G^{(n)} \rightarrow M^{(n)}]\in Rep_k(\mathit{Q}_n),\,G^{(n)} \in CoAlg(Rep_k(\mathit{Q}_n))},\,G^{(n)} \in rep_k(\mathit{Q}_n)} G^{(n)}, \,\, {\psi}'_{\!_{1}}, {\psi}'_{\!_{2}}, ..., {\psi}'_{\!_{n-1}})$ \\
  &   \\
   &   \\
  &   $ = (\mathcal{V}_1(M^{(1)}),\mathcal{V}_2(M^{(2)}), . . ., \mathcal{V}_n(M^{(n)}),{\psi}'_{\!_{1}}, {\psi}'_{\!_{2}}, ..., {\psi}'_{\!_{n-1}})$, \\
  
  &   \\
 
 \end{tabular}\\
 
for some unique $k$-linear maps ${\psi}'_{\!_{1}}, {\psi}'_{\!_{2}}, ..., {\psi}'_{\!_{n-1}}$. \\

\end{proof}
 
 \vspace{.2cm}
 
We end this paper by pointing out that the universal investigation above can be adjusted to study cofree objects in the centralizer and the center categories of $Rep_{\!_{(\mathit{Q}_1,\mathit{Q}_2,...,\mathit{Q}_n)}}$. Indeed, one can use \cite{Abdulwahid0} to study and describe them in terms of cofree objects in centralizer and  center categories of $Rep_k(\mathit{Q}_j)$,  $j \in \{1, \,\,  2, \,\, . . ., \,\, n\}$.

 \vspace{.5cm}

\begin{center}
• \textbf{Acknowledgment}
\end{center}
I would like to thank my adivsor Prof. Miodrag Iovanov, who I learned a lot from him,  for his support and his unremitting encouragement.\\

\vspace*{3mm} 
\begin{flushright}
\begin{minipage}{148mm}\sc\footnotesize

Adnan Hashim Abdulwahid\\
University of Iowa, \\
Department of Mathematics, MacLean Hall\\
Iowa City, IA, USA

{\tt \begin{tabular}{lllll}
{\it E--mail address} : &   {\color{blue} adnan-al-khafaji@uiowa.edu}\\
&  {\color{blue} adnanalgebra@gmail.com}\\
\end{tabular} }\vspace*{3mm}
\end{minipage}
\end{flushright}

\end{document}